\documentclass[12pt]{amsart}
\usepackage[top=1in, bottom=1in, left=1in, right=1in]{geometry}

\usepackage{amssymb,mathrsfs}

\newtheorem{thm}{Theorem}
\newtheorem{lemma}{Lemma}[section]
\newtheorem{cor}{Corollary}

\theoremstyle{remark}
\newtheorem{rk}{Remark}[section]

\numberwithin{equation}{section}

\newcommand{\abs}[1]{\lvert#1\rvert}
\newcommand {\SN} {{\mathbb N}}
\newcommand {\SR} {{\mathbb R}}
\newcommand {\SQ} {{\mathbb Q}}
\newcommand {\SZ} {{\mathbb Z}}

\newcommand{\be}{\begin{equation}}
\newcommand{\ee}{\end{equation}}
\newcommand{\bea}{\begin{eqnarray}}
\newcommand{\eea}{\end{eqnarray}}
\newcommand{\bv}\boldsymbol{}
\DeclareMathOperator{\vol}{Vol}

\begin{document}

\title[Localized factorizations of integers]
{Localized factorizations of integers}

\author{Dimitris Koukoulopoulos}
\address{Department of  Mathematics\\
University of Illinois at Urbana-Champaign\\
1409 West Green Street\\
Urbana\\ IL 61801\\ U.S.A.} \email{{\tt dkoukou2@math.uiuc.edu}}


\date{\today}

\maketitle

\begin{abstract} We determine the order of magnitude of
$H^{(k+1)}(x,\bv y,2\bv y)$, the number of integers $n\le x$ that
are divisible by a product $d_1\cdots d_k$ with $y_i<d_i\le 2y_i$,
when the numbers $\log y_1,\dots,\log y_k$ have the same order of
magnitude and $k\ge 2$. This generalizes a result by Kevin Ford when
$k=1$. As a corollary of these bounds, we determine the number of
elements up to multiplicative constants that appear in a
$(k+1)$-dimensional multiplication table as well as how many
distinct sums  of $k+1$ Farey fractions there are modulo 1.
\end{abstract}

\section{Introduction}\label{intro} Let $k$ be a fixed positive integer. Take all numbers up to
$N$ and form all possible products $n_1\cdots n_{k+1}$ with $n_i\le
N$ for all $i$. Obviously, there will be many numbers that appear
numerous times in this $(k+1)$-dimensional multiplication table. A
natural question arising is how many distinct integers there are in
the table. This question was first posed by Erd\H os (see
\cite{erd1}\;and \cite{erd2}) when $k=1$. Motivated by this problem
we define
$$A_{k+1}(N)=\lvert\{n_1\cdots n_{k+1}:n_i\le
N\;(1\le i\le k+1)\}\rvert.$$ The key to understanding the
combinatorics of $A_{k+1}$ is the counting function of localized
factorizations
$$H^{(k+1)}(x,\bv y,\bv z):=\lvert\{n\le x:\tau_{k+1}(n,\bv y,\bv z)\ge 1\}\rvert,$$ where
$$\tau_{k+1}(n,\bv y,\bv z):=|\{(d_1,\dots,d_k):d_1\cdots
d_k|n,\;y_i<d_i\le z_i\;(1\le i\le k)\}|$$ and $\bv y$ and $\bv z$
are $k$-dimensional vectors. The transition from $H^{(k+1)}$ to
$A_{k+1}$ is achieved via the elementary inequalities
\be\label{s1e1}\begin{split}&H^{(k+1)}\biggl(\frac{N^{k+1}}{2^k},\left(\frac N2,\dots,\frac N2\right),(N,\dots,N)\biggr)\le A_{k+1}(N)\\
&\le\sum_{\substack{1\le 2^{m_i}\le N\\1\le i\le
k}}H^{(k+1)}\biggl(\frac{N^{k+1}}{2^{m_1+\cdots+m_k}},\biggl(\frac
N{2^{m_1+1}},\dots,\frac N{2^{m_k+1}}\biggr), \biggl(\frac
N{2^{m_1}},\dots,\frac N{2^{m_k}}\biggr)\biggr).\end{split}\ee When
$k=1$, Ford \cite{kf2}, improving upon estimates of Tenenbaum
\cite{ten}, showed that
$$H^{(2)}(x,y,2y)\asymp\frac x{(\log y)^{Q(\frac1{\log2})}(\log\log
y)^{3/2}}\quad(3\le y\le\sqrt{x}),$$ where
$$Q(u):=\int_1^u\log t\,dt=u\log u-u+1\quad(u>0).$$ As a consequence, he proved that
$$A_2(N)\asymp\frac{N^2}{(\log N)^{Q(\frac1{\log2})}(\log\log N)^{3/2}}\quad(N\ge
3).$$ In the present paper we generalize this result by Ford to the
function $H^{(k+1)}$. Set $\rho=(k+1)^{1/k}$. Then we prove the
following theorem which gives the order of magnitude of
$H^{(k+1)}(x,\bv y,2\bv y)$ when all the numbers $\log
y_1,\dots,\log y_k$ have the same order of magnitude. This suffices
for the application to $A_{k+1}$.

\begin{thm}\label{th1} Let $k\ge 2$ and $0<\delta\le1$. Consider $x\ge
1$ and $3\le y_1\le y_2\le\cdots\le y_k$ with $2^{k+1}y_1\cdots
y_k\le x/y_1^\delta$. Then
$$H^{(k+1)}(x,\bv y,2\bv y)\ll_{k,\delta}\biggl(\frac{\log y_k}{\log y_1}\biggr)^{k+1}\frac
x{(\log y_1)^{Q(\frac1{\log\rho})}(\log\log y_1)^{3/2}}.$$
Furthermore, if we assume that $y_k\le y_1^c$ for some $c\ge 1$,
then
$$H^{(k+1)}(x,\bv y,2\bv
y)\gg_{k,\delta,c}\frac x{(\log y_1)^{Q(\frac1{\log\rho})}(\log\log
y_1)^{3/2}},$$ and consequently 
$$
H^{(k+1)}(x,\bv y,2\bv y)
	\asymp_{k,\delta,c} \frac x{(\log y_1)^{Q(\frac1{\log\rho})}(\log\log y_1)^{3/2}}.
$$
\end{thm}

As a corollary we obtain the order of magnitude of $A_{k+1}$ for
every fixed $k\ge2$.

\begin{cor}\label{cor1} Let $k\ge2$. For every $N\ge 3$ we have that
$$
A_{k+1}(N)\asymp_k\frac{N^{k+1}}{(\log N)^{Q(\frac1{\log\rho})}(\log\log N)^{3/2}}.
$$
\end{cor}

\begin{proof} We apply Theorem \ref{th1}\;to the left inequality of
\eqref{s1e1}\;to obtain the lower bound. For the upper bound we apply
Theorem \ref{th1}\;to the right inequality of \eqref{s1e1}\;if
$2^{m_i+1}\le\sqrt{N}$ for all $i=1,\dots,k$. Note that in this case
$N/ 2^{m_i+1} \le ( N/ 2^{m_j+1})^2$ for all $i$
and $j$ in $\{1,\dots,k\}$. Otherwise, we use the trivial bound $H^{(k+1)}(x,\bv y,2\bv y)\le x$.
\end{proof}

\medskip

{\bf Another application.} In \cite{hh}\;Haynes and Homma study the
set
$$F_R(k+1):=\biggl\{\frac{b_1}{r_1}+\cdots+\frac{b_{k+1}}{r_{k+1}}\pmod 1:1\le b_i\le r_i\le R,\;(b_i,r_i)=1\;(1\le i\le
k+1)\biggr\},$$ namely the set of distinct sums modulo 1 of $k+1$
Farey fractions of order $R$. They show that
$$|F_R(2)|\asymp\frac{R^4}{(\log R)^{Q(\frac1{\log2})}(\log\log R)^{3/2}}.$$
To estimate $|F_R(k+1)|$ for an arbitrary $k$ we need the following
theorem.

\begin{thm}\label{th2} Let $k\ge 1$, $0<\delta\le1$ and $c\ge 1$. Consider $x\ge
1$ and $3\le y_1\le y_2\le\cdots\le y_k$ with $2^{k+1}y_1\cdots
y_k\le x/y_1^\delta$ and $y_k\le y_1^c$. Then
$$\widetilde{H}^{(k+1)}(x,\bv y,2\bv y):=\sum_{\substack{x/2<n\le x,\mu^2(n)=1\\\tau_{k+1}(n,\bv y,2\bv y)\ge
1}}\frac{\phi(n)}{n}\gg_{k,\delta,c}\frac x{(\log
y_1)^{Q(\frac1{\log\rho})}(\log\log y_1)^{3/2}}.$$
\end{thm}

Observe that the above theorem is stronger than the lower bound in
Theorem \ref{th1}. As a corollary we obtain the order of magnitude
of the cardinality of $F_R(k+1)$ for every fixed $k\ge2$.

\begin{cor}\label{cor2} Let $k\ge2$. For every $R\ge3$ we have that $$|F_R(k+1)|\asymp_{k}\frac{R^{2k+2}}{(\log
R)^{Q(\frac1{\log\rho})}(\log\log R)^{3/2}}.$$
\end{cor}

\begin{proof} In Corollary 2 in \cite{hh}\;it was shown that
\be\begin{split}F_R(k+1)=\biggl\{\frac br:&1\le b\le
r,\;(b,r)=1,~r=r_1\cdots r_{k+1},\nonumber\\
&r_i\le R~(1\le i\le k+1),~(r_i,r_j)=1\;(1\le i<j\le
k+1)\biggr\}.\end{split}\nonumber\ee Therefore if we set
$$\mathscr{A}_{k+1}(N)=\{n_1\cdots n_{k+1}:n_i\le N\;(1\le i\le
k+1)\}$$ so that $A_{k+1}(N)=|\mathscr{A}_{k+1}(N)|$, then
\be\label{s1e2}|F_R(k+1)|\le\sum_{r\in\mathscr{A}_{k+1}(R)}\phi(r)\le
R^{k+1}A_{k+1}(R)\ll_k\frac{R^{2k+2}}{(\log
R)^{Q(\frac1{\log\rho})}(\log\log R)^{3/2}},\ee by Corollary
\ref{cor1}. Moreover, \be\label{s1e3}\begin{split}|F_R(k+1)|&\ge
\sum_{\substack{r\in\mathscr{A}_{k+1}(R)\\\frac{R^{k+1}}{2^{k+1}}<r\le\frac{R^{k+1}}{2^k},\mu^2(r)=1}}\phi(r)
\ge\frac{R^{k+1}}{2^{k+1}}\sum_{\substack{r\in\mathscr{A}_{k+1}(R)\\\frac{R^{k+1}}{2^{k+1}}<r\le
\frac{R^{k+1}}{2^k},\mu^2(r)=1}}\frac{\phi(r)}r\\
&\ge\frac{R^{k+1}}{2^{k+1}}
\widetilde{H}^{(k+1)}\biggl(\frac{R^{k+1}}{2^k},\left(\frac R2,\dots,\frac R2\right),(R,\dots,R)\biggr)\\
&\gg_k\frac{R^{2k+2}}{(\log R)^{Q(\frac1{\log\rho})}(\log\log
R)^{3/2}},\end{split}\ee by Theorem \ref{th2}. Combine inequalities
\eqref{s1e2}\;and \eqref{s1e3}\;to complete the proof.
\end{proof}

\medskip

{\bf Heuristic argument.} We now present a heuristic argument for
Theorem \ref{th1}, first given by Ford \cite{kf1}\;for the case
$k=1$. Before we start developing the heuristic we introduce some
notation. For $a\in\SN$ set
$$\tau_{k+1}(a)=|\{(d_1,\dots,d_k):d_1\cdots d_k|a\}|,$$
\be\begin{split}\mathcal{L}^{(k+1)}(a)=&\{\bv y\in\SR^k:\tau_{k+1}(a,(e^{y_1},\dots,e^{y_k}),2(e^{y_1},\dots,e^{y_k}))\ge 1\}\nonumber\\
=&\bigcup_{d_1\cdots d_k|a}[\log(d_1/2),\log
d_1)\times\cdots\times[\log(d_k/2),\log d_k),\end{split}\nonumber\ee
and
$$L^{(k+1)}(a)=\vol(\mathcal{L}^{(k+1)}(a)),$$ where ``$\vol$" is here the
$k$-dimensional Lebesgue measure. Assume now that $\log
y_1,\dots,\log y_k$ have the same order of magnitude. Let $n\in\SN$.
Write $n=ab$, where
$$a=\prod_{p^e\|n,p\le2y_1}p^e.$$ For simplicity assume that $a$ is
square-free and that $a\le y_1^C$ for some large constant $C$.
Consider the set
$$D_{k+1}(a)=\{(\log d_1,\dots,\log d_k):d_1\cdots
d_k|a\}.$$ If $D_{k+1}(a)$ was well-distributed in $[0,\log a]^k$,
then we would expect that
$$\tau_{k+1}(a,\bv y,2\bv
y)\approx\tau_{k+1}(a)\frac{(\log 2)^k}{(\log
a)^k}\approx\frac{(k+1)^{\omega(a)}}{(\log y_1)^k},$$ which is
$\ge1$ when $\omega(a)\ge m:=\left\lfloor\frac{\log\log
y_1}{\log\rho}\right\rfloor+O(1)$. We expect that
$$\lvert\{n\le x:\omega(a)=r\}\rvert\approx\frac x{\log
y_1}\frac{(\log\log y_1)^r}{r!}$$ (for the upper bound see Theorem
08 in \cite{Hall_Ten}). Therefore, heuristically, we should have
$$H^{(k+1)}(x,\bv y,2\bv y)\approx\frac x{\log y_1}\sum_{r\ge
m}\frac{(\log\log y_1)^r}{r!}\asymp\frac x{(\log
y_1)^{Q(\frac1{\log\rho})}(\log\log y_1)^{\frac12}}.$$ Comparing the
above estimate with Theorem \ref{th1}\;we see that we are off by a
factor of $\log\log y_1$. The problem arises from the fact that
$D_{k+1}(a)$ is usually not well-distributed, but it has many
clumps. A way to measure this is the quantity $L^{(k+1)}(a)$ defined
above. Consider $n$ with $\omega(a)=m$ and write $a=p_1\cdots p_m$
with $p_1<\cdots<p_m$. We expect that the primes $p_1,\dots,p_m$ are
uniformly distributed on a $\log\log$ scale (see chapter 1 of
\cite{Hall_Ten}), that is
$$\log\log p_j\sim j\frac{\log\log y_1}m=j\log\rho+O(1).$$ But,
by the Law of the Iterated Logarithm (see Theorem 11 in
\cite{Hall_Ten}), we expect deviations from the mean value of the
order of $\sqrt{\log\log y_1}$. In particular, with probability
tending to 1 there is a $j$ such that
$$\log\log p_j\le j\log\rho-\sqrt[3]{\log\log y_1}.$$ The elements of
$D_{k+1}(a)$ live in $(k+1)^{m-j}$ translates of the set
$D_{k+1}(p_1\cdots p_j)$. Therefore
$$L^{(k+1)}(a)\le(k+1)^{m-j}L^{(k+1)}(p_1\cdots p_j)\lesssim(k+1)^{m}\exp\{-k\sqrt[3]{\log\log y_1}\},$$ which
is much less than $\tau_{k+1}(a)=(k+1)^{m}$. So we must focus on
abnormal numbers $n$ for which
$$\log\log p_j\ge j\log\rho-O(1)\;\;\;(1\le j\le m).$$ The probability that an $n$
has this property is about $\frac1m$ (Ford, \cite{kf4}). Thus we are
led to the refined estimate that
$$H^{(k+1)}(x,\bv y,2\bv y)\asymp\frac{x}{(\log
y_1)^{Q(\frac1{\log\rho})}(\log\log y_1)^{\frac32}},$$ which turns
out to be the correct one. It is worthwhile noticing here that the
exponent $3/2$ of $\log\log y_1$ appears for the same reason for all
$k$.

\medskip

The proof of the upper bound in Theorem \ref{th1}\;is a
generalization of the methods used in \cite{kf1}\;and \cite{kf2}.
However, the methods used in these papers to get the lower bound
fail when $k>3$. To illustrate what we mean we first make some
definitions. For $\bv u=(u_1,\dots,u_k)$ set $e^{\bv
u}=(e^{u_1},\dots,e^{u_k}).$ Also, let
$$M_p(a)=\int_{\SR^k}\tau_{k+1}(a,e^{\bv u},2e^{\bv u})^pd\bv u.$$
For $\bv d=(d_1,\dots,d_k)$ let $\chi_{\bv d}$ be the characteristic
function of the $k$-dimensional cube $[\log\frac{d_1}2,\log
d_1)\times\cdots\times[\log\frac{d_k}2,\log d_k)$ and observe that
\be\label{introe1}\begin{split}M_p(a)&=\sum_{d_1\cdots
d_k|a}\int_{\SR^k}\left(\sum_{e_1\cdots e_k|a}\chi_{\bv d}(\bv
u)\chi_{\bv
e}(\bv u)\right)^{p-1}d\bv u\\
&\approx\sum_{d_1\cdots d_k|a}\int_{\SR^k}\chi_{\bv d}(\bv u)|\{\bv
e:e_1\cdots e_k|a,|\log(d_i/e_i)|\le\log2\;(1\le i\le
k)\}|^{p-1}d\bv u\\
&\approx\sum_{d_1\cdots d_k|a}|\{\bv e:e_1\cdots
e_k|a,|\log(d_i/e_i)|\le\log2\;(1\le i\le k)\}|^{p-1}.\end{split}\ee
In particular,
$$M_1(a)\approx\tau_{k+1}(a)$$ and $$M_2(a)\approx|\{(\bv d,\bv
e):d_1\cdots d_k|a,e_1\cdots e_k|a,|\log(e_i/d_i)|\le\log2\;(1\le
i\le k)\}|.$$ The main argument in \cite{kf2}\;uses the first and
second moments $M_1(a)$ and $M_2(a)$, respectively, to bound
$L^{(2)}(a)$ from below. However, when $k>3$, $M_2(a)$ is too large
and the method breaks down. This forces us to consider $p$-th
moments for $p\in(1,2)$. The problem is that, whereas $M_1(a)$ and
$M_2(a)$ have a straightforward combinatorial interpretation, as
noticed above, $M_p(a)$ does not. In a sense, the first moment
counts points in the space and the second moment counts pairs of
points. So philosophically speaking, for $p\in(1,2)$ $M_p(a)$ counts
something between single points and pairs (this is captured by the
fractional exponent $p-1$ in the right hand side of
\eqref{introe1}). To deal with this obstruction we apply H\"older's
inequality in a way that allows us to continue using combinatorial
arguments.

\medskip

{\bf Acknowledgements.} The author would like to thank his advisor
Kevin Ford for constant guidance and support and for pointing out
paper \cite{hh}, as well as for discussions that led to a
simplification of the proof of Lemma \ref{s2l1}.

\section{Preliminary Results}\label{prelim}

\textbf{Notation.} Let $\omega(n)$ denote the number of distinct
prime factors of $n$. Let $P^+(n)$ and $P^-(n)$ be the largest and
smallest prime factors of $n$, respectively. Adopt the notational
conventions $P^+(1)=0$ and $P^-(1)=\infty$. For $1\le y\le x$ we use
the standard notation $\mathscr{P}(y,x)=\{n\in\SN:p|n\Rightarrow
y<p\le x\}.$ Finally, constants implied by $\ll,\;\gg$ and $\asymp$
might depend on several parameters, which will always be specified
by a subscript.

\medskip

We need some results from number theory and analysis. We start with
a sieve estimate.

\begin{lemma}\label{s2l1}We have that
\be\label{s2e1}\lvert\{n\le x:P^-(n)>z\}\rvert\ll\frac x{\log
z}\quad(1<z\le x)\ee and \be\label{s2e2}\sum_{\substack{x/2<n\le
x,\mu^2(n)=1\\P^-(n)>z}}\frac{\phi(n)}n\gg\frac x{\log
z}\quad(2<2z\le x).\ee
\end{lemma}

\begin{proof} Inequality \eqref{s2e1}\;is a standard application
of sieve methods (see Theorem 06 in \cite{Hall_Ten}\;or Theorem 8.4
in \cite{halb}). By similar methods we may also show that
\be\label{s2e3}\lvert\{x/2<n\le
x:\mu^2(n)=1,P^-(n)>z\}\rvert\gg\frac x{\log z}\;\;\;(2<2z\le x).\ee
Moreover, by Theorem 01 in \cite{Hall_Ten}\;we have that
$$\sum_{\substack{n\le x\\P^-(n)>z}}\frac{n}{\phi(n)}\ll\frac x{\log
z}.$$ A simple application of the Cauchy-Schwarz inequality
completes the proof of \eqref{s2e2}.
\end{proof}

Moreover, we need the following estimates for certain averages of arithmetic functions that satisfy a growth condition of
multiplicative nature. The first part of the next lemma is an easy corollary of \cite[Lemma 2]{NT} (setting the implicit polynomial $Q$ to be 1). For the sake of completeness and because the authors in \cite{NT} work in much greater generality, we repeat the required part of their argument here. See also \cite{kf1} (the arXiv version, with reference number math.NT/0607473) for an alternative proof of part (b).

\renewcommand{\labelenumi}{(\alph{enumi})}

\begin{lemma}\label{s2l4} Let $f:\SN\to[0,+\infty)$ be an arithmetic function. Assume that there
exists a constant $C_f$ depending only on $f$ such that $f(ap)\le
C_ff(a)$ for all $a\in\SN$ and all primes $p$ with $(a,p)=1$.
\begin{enumerate}
\item Let $A\in\SR$ and $3/2\le y\le x\le z^C$ for some $C>0$. Then
$$
\sum_{\substack{a\in\mathscr{P}(y,x)\\\mu^2(a)=1,\, a>z}}\frac{f(a)}a\log^A(P^+(a))
	\ll_{C_f,A,C}\exp\left\{-\frac{\log z}{\log x}\right\}(\log x)^A
		\sum_{\substack{a\in\mathscr{P}(y,x)\\\mu^2(a)=1}} \frac{f(a)}a.
$$

\item Let $3/2\le y\le x$, $h\ge 0$ and $\epsilon>0$. Then
$$
\sum_{\substack{a\in\mathscr{P}(y,x)\\\mu^2(a)=1}}\frac{f(a)}{a \log^h(P^+(a)+x^\epsilon/a)}
	\ll_{C_f,h,\epsilon}\frac1{(\log x)^h}\sum_{\substack{a\in\mathscr{P}(y,x)\\\mu^2(a)=1}}\frac{f(a)}a.
$$
\end{enumerate}
\end{lemma}


\begin{proof} (a) We have that
\be\label{pr e100}\begin{split}
\sum_{\substack{a\in\mathscr{P}(y,x)\\\mu^2(a)=1,\, a>z}}\frac{f(a)}a\log^A(P^+(a))
	&= \sum_{y<p\le x} \frac{(\log p)^A}{p} \sum_{\substack{ b\in\mathscr{P}(y,p)\\ \mu^2(bp)=1, \, b>z/p}}\frac{f(bp)}{b} \\
	&\le C_f \sum_{y<p\le x} \frac{(\log p)^A}{p} 
		\sum_{\substack{ b\in\mathscr{P}(y,p)\\ \mu^2(b)=1, b>z/p}}\frac{f(b)}{b} .
\end{split}\ee
Fix some $p\in(y,x]$. Then we have that 
\be\label{pr e101}
\sum_{\substack{b \in\mathscr{P}(y,p) \\ \mu^2(b)=1,\, b>z/p }} \frac{f(b)}{b}
	\le \exp\left\{ - \frac{\log(z/p)}{\log p} \right\} 
		\sum_{\substack{b\in\mathscr{P}(y,p) \\ \mu^2(b)=1}}\frac{f(b)}{b^{1-1/ \log p }}  .
\ee
We write $n^{1/\log p}=(1*g)(n)$, where $g(p')=(p')^{1/\log p}-1 \ll \log p' / \log p $ for every prime $p'\le p$. Therefore
\begin{align*}
\sum_{\substack{b\in\mathscr{P}(y,p)\\ \mu^2(b)=1}}\frac{f(b)}{b^{1-1/ \log p }}
	&= \sum_{\substack{b\in\mathscr{P}(y,p)\\ \mu^2(b)=1}}\frac{f(b)}{b}
		\sum_{d|b} g(d) 
		=\sum_{\substack{c,d\in\mathscr{P}(y,p)\\ \mu^2(cd)= 1 }} \frac{f(cd) g(d)}{cd}  \\
	&\le \sum_{\substack{c,d\in\mathscr{P}(y,p)\\ \mu^2(cd)= 1 }} \frac{f(c) C_f^{\omega(d)} g(d) }{cd} 
	\le \sum_{\substack{c \in\mathscr{P}(y,p)\\ \mu^2(c)= 1 }} \frac{f(c) }{c} 
		\prod_{y<p'\le p} \left( 1+ \frac{C_f g(p')}{p'} \right) .
\end{align*}
Since
\[
\log \prod_{y<p'\le p} \left( 1+ \frac{C_f g(p')}{p'} \right)  
	\le \sum_{p'\le p} \frac{ C_f g(p') }{p'}
	\ll \sum_{p'\le p} \frac{C_f \log p'}{p'\log p} \ll C_f,
\]
we conclude that
\[
\sum_{\substack{b\in\mathscr{P}(y,p)\\ \mu^2(b)=1}}\frac{f(b)}{b^{1-1/ \log p }}
	\le e^{O(C_f)} \sum_{\substack{b \in\mathscr{P}(y,p)\\ \mu^2(b)= 1 }} \frac{f(b) }{b}  .
\]
Inserting this inequality into \eqref{pr e101}, we arrive to the estimate
\[
\sum_{\substack{a\in\mathscr{P}(y,p)\\\mu^2(a)=1,\, a> z/p }}\frac{f(a)}{a}
	\ll_{C_f} \exp\left\{ - \frac{\log z}{\log p}\right\} 
		\sum_{\substack{a \in\mathscr{P}(y,p)\\ \mu^2(a)= 1 }} \frac{f(a) }{a} ,
\]
for every prime $p\in(y,x]$. Thus \eqref{pr e100} becomes
\[
\sum_{\substack{a\in\mathscr{P}(y,x)\\\mu^2(a)=1,\, a>z}}\frac{f(a)}a\log^A(P^+(a))
	\ll_{C_f} \sum_{y<p\le x} \frac{(\log p)^A}{p} \exp\left\{ - \frac{\log z}{\log p}\right\}
		\sum_{\substack{a \in\mathscr{P}(y,p)\\ \mu^2(a)= 1 }} \frac{f(a) }{a} .
\]
Finally, note that  $\mu := z^{1/\log x} \ge e^{1/C}>1$. Hence
\begin{align*}
\sum_{p\le x}\frac{(\log p)^A}{p}  \exp\left\{-\frac{\log z}{\log p}\right\}
	\le\sum_{0\le n\le \frac{\log x }{\log 2} } \frac{1}{\mu^{2^n}}
	\sum_{ \frac{\log x}{2^{n+1}}< \log p \le \frac{\log x}{2^n} } \frac{(\log p)^A}p   
	&\ll_A (\log x)^A\sum_{n=0}^\infty\frac1{\mu^{2^n} 2^{An} } \\
	&\ll_{A,C} \frac{(\log x)^A }{\mu}  ,
\end{align*}
and part (a) follows.

\medskip

(b) Write
$$\sum_{\substack{a\in\mathscr{P}(y,x)\\\mu^2(a)=1}}\frac{f(a)}{a\log^h(P^+(a)+x^\epsilon/a)}=T_1+T_2,$$ where
$$T_1=\sum_{\substack{a\in\mathscr{P}(y,x)\\\mu^2(a)=1,a\le
x^{\epsilon/2}}}\frac{f(a)}{a\log^h(P^+(a)+x^\epsilon/a)}\quad{\rm
and}\quad
T_2=\sum_{\substack{a\in\mathscr{P}(y,x)\\\mu^2(a)=1,a>x^{\epsilon/2}}}\frac{f(a)}{a\log^h(P^+(a)+x^\epsilon/a)}.$$
Then clearly $$T_1\ll_{h,\epsilon}(\log
x)^{-h}\sum_{\substack{a\in\mathscr{P}(y,x)\\\mu^2(a)=1}}\frac{f(a)}a$$
as well as
$$T_2\le\sum_{\substack{a\in\mathscr{P}(y,x)\\\mu^2(a)=1,a>x^{\epsilon/2}}}\frac{f(a)}a\log^{-h}(P^+(a))
\ll_{C_f,h,\epsilon}(\log
x)^{-h}\sum_{\substack{a\in\mathscr{P}(y,x)\\\mu^2(a)=1}}\frac{f(a)}a,$$
by part (a), and the desired result follows.
\end{proof}

Finally, we need a covering lemma which is a slightly different
version of Lemma 3.15 in \cite{fol}. If $r$ is a positive real
number and $I$ is a $k$-dimensional rectangle, then we denote with
$r I$ the rectangle which has the same center with $I$ and $r$ times
its diameter. More formally, if $\bv{x_0}$ is the center of $I$,
then $r I:=\{r(\bv x-\bv{x_0})+\bv{x_0}:\bv x\in I\}.$ The lemma is
then formulated as follows.

\begin{lemma}\label{s2l3} Let $I_1,...,I_N$ be
$k$-dimensional cubes of the form
$[a_1,b_1)\times\cdots\times[a_k,b_k)$ $(b_1-a_1=\cdots=b_k-a_k>0)$.
Then there exists a sub-collection $I_{i_1},\dots,I_{i_M}$ of
mutually disjoint cubes such that
$$\bigcup_{n=1}^NI_n\subset\bigcup_{m=1}^M3I_{i_m}.$$
\end{lemma}

\begin{rk}The above lemma is very useful in the following sense. If
$A=\bigcup_{i=1}^NI_i\subset\SR^k$ with $I_i$ as in Lemma
\ref{s2l3}, then in order to control a sum of the form
$$\sum_{\substack{(\log p_1,\dots,\log p_k)\in
A\\p_j\;\text{prime}}}f(p_1,\dots,p_k),$$ it suffices to estimate
sums of the form
$$\sum_{e^{a_1}\le p_1<e^{b_1}}\cdots\sum_{e^{a_k}\le
p_k<e^{b_k}}f(p_1,\dots,p_k),$$ which are much easier to handle.
\end{rk}

\section{Lower bounds}\label{lb}

Before we launch into the lower bounds proof we list some
inequalities about $L^{(k+1)}$.

\begin{lemma}\label{s3l1}
\begin{enumerate}
\item$L^{(k+1)}(a)\le\min\{\tau_{k+1}(a)(\log2)^k,(\log a+\log2)^k\}$.\\
\item If $(a,b)=1$, then $L^{(k+1)}(ab)\le\tau_{k+1}(a)L^{(k+1)}(b)$.\\
\item If $p_1,\dots,p_m$ are distinct prime numbers, then
$$L^{(k+1)}(p_1\cdots p_m)\le\min_{1\le j\le m}\{(k+1)^{m-j}(\log(p_1\cdots
p_j)+\log 2)^k\}.$$
\end{enumerate}
\end{lemma}

\begin{proof} The proof is very similar to the proof of Lemma
3.1 in \cite{kf1}.
\end{proof}

\renewcommand{\labelenumi}{(\arabic{enumi})}

For the rest of this section we assume that $y_1>C_1$, where $C_1$
is a large enough positive constant, possibly depending on
$k,\delta$ and $c$; for if $y_1\le C_1$, then
$\widetilde{H}^{(k+1)}(x,\bv y,2\bv y)\gg_{C_1}x$ and Theorem
\ref{th2}\;follows immediately. The value of the constant $C_1$ will
not be specified, but it can be computed effectively if one goes
through the proof.

\medskip

We now prove the following lemma which is the starting point to
obtain a lower bound for $\widetilde{H}^{(k+1)}(x,\bv y,2\bv y)$.
Note that it is similar to Lemma 2.1 in \cite{kf1}\;and Lemma 4.1 in
\cite{kf2}.

\begin{lemma}\label{s3l2} Let $k\ge 1$, $0<\delta<1$ and $c\ge 1$. Then for $x\ge 1$
and $3\le y_1\le y_2\le\cdots\le y_k\le y_1^c$ with
$2^{k+1}y_1\cdots y_k\le x/y_1^\delta$ we have that
\be\label{s3e1}\widetilde{H}^{(k+1)}(x,\bv y,2\bv
y)\gg_{k,\delta,c}\frac{x}{(\log
y_1)^{k+1}}\sum_{\substack{P^+(a)\le
y_1\\\mu^2(a)=1}}\frac{\phi(a)}a\frac{L^{(k+1)}(a)}a.\ee
\end{lemma}

\begin{proof} Set $x^\prime=x/(2^{k+1}y_1\cdots y_k).$ Consider squarefree integers $n=ap_1\cdots
p_kb\in(x/2,x]$ such that

\renewcommand{\labelenumi}{(\arabic{enumi})}

\begin{enumerate}\item $a\le y_1^{\delta/4}$;
\item $p_1,\dots,p_k$ are distinct prime numbers with
$(\log(y_1/p_1),\dots,\log(y_k/p_k))\in\mathcal{L}^{(k+1)}(a)$;
\item $P^-(b)>y_1^{\delta/4}$ and $b$ has at most one prime factor in
$(y_1^{\delta/4},2y_1^c]$.
\end{enumerate}

\renewcommand{\labelenumi}{(\alph{enumi})}

Condition (2) is equivalent to the existence of positive integers
$d_1,\dots,d_k$ such that $d_1\cdots d_k|a$ and $y_i/p_i<d_i\le2
y_i/p_i$, $i=1,\dots k$. In particular, $\tau_{k+1}(n,\bv y,2\bv
y)\ge 1$. Furthermore,
$$y_1^{1-\delta/4}\le\frac
{y_1}a\le\frac{y_i}{d_i}<p_i\le2\frac{y_i}{d_i}\le2y_1^c.$$ Hence
this representation of $n$, if it exists, is unique up to a possible
permutation of $p_1,\dots,p_k$ and the prime factors of $b$ lying in
$(y_1^{\delta/4},2y_1^c]$. Since $b$ has at most one prime factor in
$(y_1^{\delta/4},2y_1^c]$, $n$ has a bounded number of such
representations. Fix $a$ and $p_1,\dots,p_k$ and note that
\be\label{s3e1b}X:=\frac x{ap_1\cdots p_k}\ge\frac
x{y_1^{\delta/4}2^ky_1\cdots y_k}\ge2y_1^{3\delta/4}.\ee Therefore
$$\sum_{b\;\text{admissible}}\frac{\phi(b)}b\ge\frac12\biggl(\sum_{\substack{X/2<b\le
X\\P^-(b)>2y_1^c,\mu^2(b)=1}}\frac{\phi(b)}b+\sum_{\substack{X/2<p\le
X\\p\notin\{p_1,\dots,p_k\}}}\frac{\phi(p)}p\biggr)\gg_{c,\delta}\frac
X{\log y_1},$$ by~\eqref{s3e1b}, Lemma~\ref{s2l1} and the Prime
Number Theorem. Hence \be\label{s3e2}\widetilde{H}^{(k+1)}(x,\bv
y,2\bv y)\gg_{k,\delta,c}\frac{x}{\log y_1}\sum_{\substack{a\le
y_1^{\delta/4}\\\mu^2(a)=1}}\frac{\phi(a)}{a^2}
\sum_{\substack{(\log\frac{y_1}{p_1},...,\log\frac{y_k}{p_k})\in\mathcal{L}^{(k+1)}(a)\\p_1,\dots,p_k\;\text{distinct}}}
\frac{\phi(p_1)\cdots\phi(p_k)}{p_1^2\cdots p_k^2}.\ee Fix $a\le
y_1^{\delta/4}$. Let $\{I_r\}_{r=1}^R$ be the collection of cubes
$[\log(d_1/2),\log d_1)\times\cdots\times[\log(d_k/2),\log d_k)$ for
$d_1\cdots d_k|a$. Then for $I=[\log(d_1/2),\log
d_1)\times\cdots\times[\log(d_k/2),\log d_k)$ in this collection we
have
\be\begin{split}\sum_{(\log\frac{y_1}{p_1},...,\log\frac{y_k}{p_k})\in
I}\frac{\phi(p_1)\cdots\phi(p_k)}{p_1^2\cdots
p_k^2}=\prod_{i=1}^k\sum_{y_i/d_i<p_i\le2y_i/d_i}\frac{\phi(p_i)}{p_i^2}&\gg_k\prod_{i=1}^k\frac1{\log(\max\{2,2y_i/d_i\})}\nonumber\\
&\gg_k\frac1{\log y_1\cdots\log y_k}\end{split}\nonumber\ee because
$d|a$ implies that $d\le y_1^{\delta/4}$. Similarly,
\be\begin{split}\sum_{\substack{(\log\frac{y_1}{p_1},...,\log\frac{y_k}{p_k})\in
I\\p_1,\dots,p_k\;\text{not
distinct}}}\frac{\phi(p_1)\cdots\phi(p_k)}{p_1^2\cdots
p_k^2}&\le\sum_{1\le i<j\le
k}\prod_{\substack{t=1\\t\notin\{i,j\}}}^k\sum_{y_t/d_t<p_t\le2y_t/d_t}\frac1{p_t}\sum_{p_i>\max\{y_i,y_j\}/y_1^{\delta/4}}\frac1{p_i^2}\nonumber\\
&\ll_k\frac1{y_1^{1/2}}\frac1{\log y_1\cdots\log
y_k}.\end{split}\nonumber\ee Thus
$$\sum_{\substack{(\log\frac{y_1}{p_1},...,\log\frac{y_k}{p_k})\in
I\\p_1,\dots,p_k\;\text{distinct}}}\frac{\phi(p_1)\cdots\phi(p_k)}{p_1^2\cdots
p_k^2}\gg_k\frac1{\log y_1\cdots\log y_k}\gg_{k,c}\frac1{(\log
y_1)^k},$$ provided that $C_1$ is large enough. By Lemma \ref{s2l3},
there exists a sub-collection $\{I_{r_s}\}_{s=1}^S$ of mutually
disjoint cubes such that
$$S(\log2)^k=\vol\left(\bigcup_{s=1}^SI_{r_s}\right)\ge\frac
1{3^k}{\rm
Vol}\left(\bigcup_{r=1}^RI_r\right)=\frac{L^{(k+1)}(a)}{3^k}.$$ Hence
$$\sum_{\substack{(\log\frac{y_1}{p_1},...,\log\frac{y_k}{p_k})\in\mathcal{L}^{(k+1)}(a)\\p_1,\dots,p_k\;\text{distinct}}}
\frac{\phi(p_1)\cdots\phi(p_k)}{p_1^2\cdots
p_k^2}\ge\sum_{s=1}^S\sum_{\substack{(\log\frac{y_1}{p_1},...,\log\frac{y_k}{p_k})\in
I_{r_s}\\p_1,\dots,p_k\;\text{distinct}}}\frac{\phi(p_1)\cdots\phi(p_k)}{p_1^2\cdots
p_k^2}\gg_{k,c}\frac{L^{(k+1)}(a)}{(\log y_1)^k},$$ which together
with \eqref{s3e2}\;implies that
\be\label{s3l1e1}\widetilde{H}^{(k+1)}(x,\bv y,2\bv
y)\gg_{k,\delta,c}\frac{x}{(\log y_1)^{k+1}}\sum_{\substack{a\le
y_1^{\delta/4}\\\mu^2(a)=1}}\frac{\phi(a)}a\frac{L^{(k+1)}(a)}a.\ee
Note that the arithmetic function $a\to L^{(k+1)}(a)\phi(a)/a$
satisfies the hypothesis of Lemma \ref{s2l4} with $C_f=k+1$, by Lemma \ref{s3l1}(b). 
Hence if $M=M(k)$ is sufficiently large, then
\be\begin{split}
\sum_{\substack{a\le y_1^{\delta/4}\\\mu^2(a)=1}}\frac{\phi(a)}a\frac{L^{(k+1)}(a)}a
	&\ge\sum_{\substack{P^+(a)\le y_1^{\delta/M}\\\mu^2(a)=1,\, a\le y_1^{\delta/4}}}\frac{\phi(a)}a
		\frac{L^{(k+1)}(a)}a\nonumber\\
	&=\sum_{\substack{P^+(a)\le y_1^{\delta/M}\\\mu^2(a)=1}}\frac{\phi(a)}a\frac{L^{(k+1)}(a)}a
		\left(1+O_k(e^{-M/8})\right)\\
	&\ge\frac12\sum_{\substack{P^+(a)\le y_1^{\delta/M}\\\mu^2(a)=1}}\frac{\phi(a)}a\frac{L^{(k+1)}(a)}a,
\end{split}\ee
by Lemma \ref{s2l4}(a). So \be\begin{split}\sum_{\substack{P^+(a)\le
y_1\\\mu^2(a)=1}}\frac{\phi(a)}a\frac{L^{(k+1)}(a)}a&\le\sum_{\substack{P^+(a_1)\le
y_1^{\delta/M}\\\mu^2(a_1)=1}}\frac{\phi(a_1)}{a_1}\frac{L^{(k+1)}(a_1)}{a_1}
\sum_{\substack{a_2\in\mathscr{P}(y_1^{\delta/M},y_1)\\\mu^2(a_2)=1}}\frac{\tau_{k+1}(a_2)\phi(a_2)}{a_2^2}\nonumber\\
&\le2\sum_{\substack{a_1\le
y_1^{\delta/4}\\\mu^2(a_1)=1}}\frac{\phi(a_1)}{a_1}\frac{L^{(k+1)}(a_1)}{a_1}\prod_{y_1^{\delta/M}<p\le
y_1}\left(1+\frac{(k+1)(p-1)}{p^2}\right)\\
&\ll_{k,\delta}\sum_{\substack{a_1\le
y_1^{\delta/4}\\\mu^2(a_1)=1}}\frac{\phi(a_1)}{a_1}\frac{L^{(k+1)}(a_1)}{a_1},\end{split}\ee
where we used Lemma \ref{s3l1}(b). Inserting the above estimate into
\eqref{s3l1e1}\;completes the proof.
\end{proof}

Given $P\in(1,+\infty)$ and $a\in\SN$ set
$$W_{k+1}^P(a)=\sum_{d_1\cdots
d_k|a}\left\lvert\left\{(d_1',\dots,d_k')\in\SN^k:d_1'\cdots
d_k'|a,\;\left\lvert\log\frac{d_i'}{d_i}\right\rvert<\log2 \;(1\le
i\le k)\right\}\right\rvert^{P-1}.$$

\begin{lemma}\label{s3l4} Let $\mathcal{A}$ be a finite set of positive integers and $P\in(1,+\infty)$.
Then
$$\sum_{a\in\mathcal{A}}\frac{\tau_{k+1}(a)}a\le\biggl(\frac1{(\log2)^k}\sum_{a\in\mathcal{A}}\frac{L^{(k+1)}(a)}a\biggr)^{1-1/P}
\biggl(\sum_{a\in\mathcal{A}}\frac{W_{k+1}^P(a)}a\biggr)^{1/P}.$$
\end{lemma}

\begin{proof} For $\bv d=(d_1,\dots, d_k)\in\SR^k$ let $\chi_{\bv d}$ be the characteristic
function of the $k$-dimensional cube $[\log(d_1/2),\log
d_1)\times\cdots\times[\log(d_k/2),\log d_k)$. Then it is easy to
see that $$\tau_{k+1}(a,e^{\bv u},2e^{\bv u})=\sum_{d_1\cdots
d_k|a}\chi_{\bv d}(\bv u)$$ for all $a\in\SN$, where $e^{\bv
u}=(e^{u_1},\dots,e^{u_k})$ for $\bv u=(u_1,\dots,u_k)\in\SR^k$.
Hence
$$\int_{\SR^k}\tau_{k+1}(a,e^{\bv u},2e^{\bv u})d\bv
u=\tau_{k+1}(a)(\log2)^k$$ and a double application of H\"older's
inequality yields
\be\label{s3l4e2}(\log2)^k\sum_{a\in\mathcal{A}}\frac{\tau_{k+1}(a)}a
\le\biggl(\sum_{a\in\mathcal{A}}\frac1a\int_{\SR^k}\tau_{k+1}(a,e^{\bv
u},2e^{\bv u})^Pd\bv
u\biggr)^{1/P}\biggl(\sum_{a\in\mathcal{A}}\frac{L^{(k+1)}(a)}a\biggr)^{1-1/P}.\ee
Finally, note that \be\begin{split}\tau_{k+1}(a,e^{\bv u},2e^{\bv
u})^P=\sum_{\substack{d_1\cdots d_k|a\\u_i<\log d_i\le
u_i+\log2\\1\le i\le k}}\tau_{k+1}(a,e^{\bv u},2e^{\bv
u})^{P-1}&=\sum_{\substack{d_1\cdots d_k|a\\u_i<\log d_i\le
u_i+\log2\\1\le i\le k}}\left(\sum_{\substack{d_1'\cdots
d_k'|a\\u_i<\log d_i'\le u_i+\log2\\1\le i\le
k}}1\right)^{P-1}\\
&\le\sum_{\substack{d_1\cdots d_k|a\\u_i<\log d_i\le u_i+\log2\\1\le
i\le k}}\left(\sum_{\substack{d_1'\cdots
d_k'|a\\|\log(d_i'/d_i)|<\log 2\\1\le i\le
k}}1\right)^{P-1}.\nonumber\end{split}\ee So
$$\int_{\SR^k}\tau(a,e^{\bv u},2e^{\bv u})^Pd\bv u\le\sum_{d_1\cdots
d_k|a}\left(\sum_{\substack{d_1'\cdots d_k'|a\\|\log(d_i'/d_i)|<\log
2\\1\le i\le k}}1\right)^{P-1}\int_{\SR^k}\chi_{\bv d}(\bv u)d\bv
u=(\log2)^kW_{k+1}^P(a),$$ which together with
\eqref{s3l4e2}\;completes the proof of the lemma.
\end{proof}

Our next goal is to estimate
$$\sum_{a\in\mathcal{A}}\frac{W_{k+1}^P(a)}a$$ for suitably chosen
sets $\mathcal{A}$. This will be done in Lemmas \ref{s3l6a},
\ref{s3l6b}\;and \ref{s3l7}. First, we introduce some notation.

We generalize the construction of a sequence of primes
$\lambda_1,\lambda_2,...$ found in \cite{kf1}\;and \cite{kf2}. Set
$\lambda_0=\min\{p\;\text{prime}:p\ge k+1\}-1$. Then define
inductively $\lambda_j$ as the largest prime such that
\be\label{s3e3}\sum_{\lambda_{j-1}<p\le\lambda_j}\frac1p\le\log\rho.\ee
Note that $1/(\lambda_0+1)\le1/(k+1)<\log\rho$ because
$(k+1)\log\rho=\frac{k+1}k\log(k+1)$ is an increasing function of
$k$ and $\log4>1$. Thus the sequence $\{\lambda_j\}_{j=1}^\infty$ is
well-defined. Set
$$D_j=\{p\;\text{prime}:\lambda_{j-1}<p\le \lambda_j\}$$ and $\log\lambda_j=\rho^{\mu_j}$ for all $j\in\SN$.
Then we have the following lemma.

\begin{lemma}\label{s3l3}There exists a constant $\ell_k$ such that
$$\mu_j=j+\ell_k+O_k(\rho^{-j})\quad(j\in\SN).$$ In particular, there exists a
positive integer $L_k$ such that $$|\mu_j-j|\le L_k\quad(j\in\SN).$$
\end{lemma}

\begin{proof} By the Prime Number Theorem with de la Valee
Poussin error term \cite[p. 111]{dav}, there exists some positive
constant $c_1$ such that
\be\label{s3e4}\log_2\lambda_j-\log_2\lambda_{j-1}=\log\rho+O(e^{-c_1\sqrt{\log\lambda_{j-1}}})\ee
for all $j\in\SN$. In addition, $\lambda_j\to\infty$ as
$j\to\infty$, by construction. So if we fix
$\rho^\prime\in(1,\rho)$, then \eqref{s3e4}\;implies that
$$\frac{\log\lambda_j}{\log\lambda_{j-1}}\ge\rho^\prime$$ for sufficiently large $j$,
which in turn implies that that the series
$\sum_je^{-c_1\sqrt{\log\lambda_j}}$ converges. Thus, telescoping
the summation of \eqref{s3e4}\;yields that $\mu_j=j+O_k(1)$, and
hence $\log\lambda_j\gg_k\rho^{j}$. Summing \eqref{s3e4} again, this
time for $j=r+1,\dots,s$, we get that
$$\mu_s-\mu_r=s-r+O_k(\rho^{-r}).$$ Hence
$\{\mu_j-j\}$ is a Cauchy sequence and so it converges, say to
$\ell_k$. Finally, letting $s\to\infty$ gives us the desired result.
\end{proof}

\medskip

For $R\in\SN$ let
$$\mathscr{P}_R=\{(Y_1,\dots,Y_k):Y_i\subset\{1,\dots,R\},Y_i\cap
Y_j=\emptyset\;\text{if}\;i\neq j\}.$$ Also, for $\bv
Y=(Y_1,\dots,Y_k)\in\mathscr{P}_R$ and $\bv
I=(I_1,\dots,I_k)\in\{0,1,\dots,R\}^k$ set $$M_R(\bv Y;\bv
I)=\left\lvert\left\{(Z_1,\dots,Z_k)\in\mathscr{P}_R:\bigcup_{r=j}^k\left(Z_r\cap(I_j,R]\right)
=\bigcup_{r=j}^k\left(Y_r\cap(I_j,R]\right)\;(1\le j\le
k)\right\}\right\rvert.$$ For $\bv
b=(b_1,\dots,b_H)\in(\SN\cup\{0\})^H$ let $\mathcal{A}(\bv b)$ be
the set of square-free integers composed of exactly $b_j$ prime
factors from $D_j$ for each $j$. Set $B=b_1+\cdots+b_H$, $B_0=0$ and
$B_i=b_1+\cdots+b_i$ for all $i=1,\dots,H$. For
$I\in\{0,1,\dots,B\}$ define $E_{\bv b}(I)$ by $B_{E_{\bv
b}(I)-1}<I\le B_{E_{\bv b}(I)}$ if $I>0$ and set $E_{\bv b}(I)=0$ if
$I=0$. Lastly, for $\{X_j\}_{j\in J}$ a family of sets define
$$\mathcal{U}(\{X_j:j\in J\})=\left\{x\in\bigcup_{j\in J}X_j:|\{i\in J:x\in
X_i\}|=1\right\}.$$ In particular,
$$\mathcal{U}(\{X_1,X_2\})=X_1\triangle X_2,$$ the symmetric
difference of $X_1$ and $X_2$, and
$$\mathcal{U}(\emptyset)=\emptyset.$$

\begin{rk}\label{rk1}Assume that $Y_1,\dots,Y_n$ and $Z_1,\dots,Z_n$ satisfy $Y_i\cap Y_j=Z_i\cap Z_j=\emptyset$ for $i\neq j$.
Then the condition
$$\mathcal{U}(\{Y_j\triangle Z_j:1\le j\le n\})=\emptyset$$ is equivalent to
$$\bigcup_{j=1}^nY_j=\bigcup_{j=1}^nZ_j.$$
\end{rk}

\begin{lemma}\label{s3l6a} Let $k\ge1$, $P\in(1,2]$ and $\bv b=(b_1,\dots,b_H)\in(\SN\cup\{0\})^H$.
Then $$\sum_{a\in\mathcal{A}(\bv
b)}\frac{W_{k+1}^P(a)}a\ll_k\frac{(\log\rho)^B}{b_1!\cdots
b_H!}\sum_{0\le I_1,\dots,I_k\le
B}\prod_{j=1}^k(\rho^{P-1})^{-E_{\bv b}(I_j)}\sum_{\bv
Y\in\mathscr{P}_B}\left(M_B(\bv Y;\bv I)\right)^{P-1}.$$
\end{lemma}

\begin{proof} Let $a=p_1\cdots p_B\in\mathcal{A}(\bv b)$,
where \be\label{s3e5}p_{B_{i-1}+1},\dots,p_{B_i}\in D_i\quad(1\le
i\le H),\ee and the primes in each interval $D_j$ for $j=1,\dots,H$
are unordered. Observe that, since $p_1\cdots p_B$ is square-free
and has precisely $B$ prime factors, the $k$-tuples
$(d_1,\dots,d_k)\in\SN^k$ with $d_1\cdots d_k|p_1\cdots p_B$ are in
one to one correspondence with $k$-tuples
$(Y_1,\dots,Y_k)\in\mathscr{P}_B$; this correspondence is given by
$$d_j=\prod_{i\in Y_j}p_i\quad(1\le j\le k).$$ Using this observation twice we find that
\be\begin{split}W_{k+1}^P(p_1\cdots
p_B)&=\sum_{(Y_1,\dots,Y_k)\in\mathscr{P}_B}\left\lvert \left\{(d_1',\dots,d_k')\in\SN^k:d_1'\cdots
d_k'|a,   \right.\right.\\
&\qquad\qquad\qquad\qquad\quad
	\left.\left.	\left\lvert\log d_j'-\sum_{i\in Y_j}\log p_i\right\rvert<\log2\;(1\le j\le k)\right\}\right\rvert^{P-1}\nonumber\\
&=\sum_{(Y_1,\dots,Y_k)\in\mathscr{P}_B}
	\left(\sum_{\substack{(Z_1,\dots,Z_k)\in\mathscr{P}_B \\  \eqref{s3e10a}  }}1 \right)^{P-1},
\end{split}\ee
where for two $k$-tuples $(Y_1,\dots,Y_k)$ and $(Z_1,\dots,Z_k)$ in
$\mathscr{P}_B$ condition \eqref{s3e10a}\;is defined by
\be\label{s3e10a}-\log 2<\sum_{i\in Y_j}\log p_i-\sum_{i\in Z_j}\log
p_i<\log 2\quad(1\le j\le k).\ee Moreover, every integer
$a\in\mathcal{A}(\bv b)$ has exactly $b_1!\cdots b_H!$
representations of the form $a=p_1\cdots p_B$, corresponding to the
possible permutations of the primes $p_1,\dots,p_B$ under condition
\eqref{s3e5}. Thus \be\begin{split}\sum_{a\in\mathcal{A}(\bv
b)}\frac{W_{k+1}^P(a)}a&=\frac1{b_1!\cdots
b_H!}\sum_{\substack{p_1,\dots,p_B\\\eqref{s3e5}}}\frac1{p_1\cdots
p_B}\sum_{(Y_1,\dots,Y_k)\in\mathscr{P}_B}\left(\sum_{\substack{(Z_1,\dots,Z_k)\in\mathscr{P}_B\\\eqref{s3e10a}}}1\right)^{P-1}\nonumber\\
&=\frac1{b_1!\cdots
b_H!}\sum_{(Y_1,\dots,Y_k)\in\mathscr{P}_B}\sum_{\substack{p_1,\dots,p_B\\\eqref{s3e5}}}\frac1{p_1\cdots
p_B}
\left(\sum_{\substack{(Z_1,\dots,Z_k)\in\mathscr{P}_B\\\eqref{s3e10a}}}1\right)^{P-1}\\
&\le\frac1{b_1!\cdots
b_H!}\sum_{(Y_1,\dots,Y_k)\in\mathscr{P}_B}\biggl(\sum_{\substack{p_1,\dots,p_B\\\eqref{s3e5}}}
\frac1{p_1\cdots p_B}\biggr)^{2-P}\\
&\quad\quad\quad\quad\times\biggl(\sum_{\substack{p_1,\dots,p_B\\\eqref{s3e5}}}\frac1{p_1\cdots
p_B}\sum_{\substack{(Z_1,\dots,Z_k)\in\mathscr{P}_B\\\eqref{s3e10a}}}1\biggr)^{P-1},\end{split}\ee
by H\"older's inequality if $P<2$ and trivially if $P=2$. Observe
that $$\sum_{\substack{p_1,\dots,p_B\\\eqref{s3e5}}}\frac1{p_1\cdots
p_B}\le\prod_{j=1}^H\left(\sum_{p\in D_j}\frac
1p\right)^{b_j}\le(\log\rho)^B,$$ by \eqref{s3e3}. Consequently,
\be\label{s3e7}\begin{split}\sum_{a\in\mathcal{A}(\bv
b)}\frac{W_{k+1}^P(a)}a&\le\frac{(\log\rho)^{(2-P)B}}{b_1!\cdots
b_H!}\sum_{(Y_1,\dots,Y_k)\in\mathscr{P}_B}
\biggl(\sum_{\substack{p_1,\dots,p_B\\\eqref{s3e5}}}\frac1{p_1\cdots
p_B}\sum_{\substack{(Z_1,\dots,Z_k)\in\mathscr{P}_B\\\eqref{s3e10a}}}1\biggr)^{P-1}\\
&=\frac{(\log\rho)^{(2-P)B}}{b_1!\cdots
b_H!}\sum_{(Y_1,\dots,Y_k)\in\mathscr{P}_B}
\biggl(\sum_{(Z_1,\dots,Z_k)\in\mathscr{P}_B}\sum_{\substack{p_1,\dots,p_B\\\eqref{s3e5},\eqref{s3e10a}}}\frac1{p_1\cdots
p_B}\biggr)^{P-1}.\end{split}\ee Next, we fix
$(Y_1,\dots,Y_k)\in\mathscr{P}_B$ and
$(Z_1,\dots,Z_k)\in\mathscr{P}_B$ and proceed to the estimation of
the sum
$$\sum_{\substack{p_1,\dots,p_B\\\eqref{s3e5},\eqref{s3e10a}}}\frac1{p_1\cdots
p_B}.$$ Note that \eqref{s3e10a}\;is equivalent to
\be\label{s3e10}-\log 2<\sum_{i\in Y_j\setminus Z_j}\log
p_i-\sum_{i\in Z_j\setminus Y_j}\log p_i<\log 2\;\;\;(1\le j\le
k).\ee Conditions \eqref{s3e10}, $1\le j\le k$, are a system of $k$
inequalities. For every $j\in\{1,\dots,k\}$ and every $I_j\in
Y_j\triangle Z_j$ \eqref{s3e10}\;implies that
$p_{I_j}\in[X_j,4X_j]$, where $X_j$ is a constant depending only on
the primes $p_i$ for $i\in Y_j\triangle Z_j\setminus\{I_j\}$. In
order to exploit this simple observation to its full potential we
need to choose $I_1,\dots,I_k$ as large as possible. After this is
done, we fix the primes $p_i$ for
$i\in\{1,\dots,B\}\setminus\{I_1,\dots,I_k\}$ and estimate the sum
over $p_{I_1},\dots,p_{I_k}$. The obvious choice is to set $I_j=\max
Y_j\triangle Z_j$, $1\le j\le k$. However, in this case the indices
$I_1,\dots,I_k$ and the numbers $X_1,\dots,X_k$ might be
interdependent in a complicated way, which would make the estimation
of the sum over $p_{I_1},\dots,p_{I_k}$ very hard. So it is
important to choose large $I_1,\dots,I_k$ for which at the same time
the dependence of $X_1,\dots,X_k$ is simple enough to allow the
estimation of the sum over $p_{I_1},\dots,p_{I_k}$. What we will do
is to construct large $I_1,\dots,I_k$ such that if we fix the primes
$p_i$ for $i\in\{1,\dots,B\}\setminus\{I_1,\dots,I_k\}$, then
\eqref{s3e10}\;becomes a linear system of inequalities with respect
to $\log p_{I_1},\dots,\log p_{I_k}$ that corresponds to a
triangular matrix and hence is easily solvable (actually, we have to
be slightly more careful, but this is the main idea).

Define $I_1,\dots,I_k$ and $m_1,\dots,m_k$ with
$I_i\in(Y_{m_i}\triangle Z_{m_i})\cup\{0\}$ for all
$i\in\{1,\dots,k\}$ inductively, as follows. Let
$$I_1=\max\left\{\mathcal{U}(Y_1\triangle Z_1,\dots,Y_k\triangle Z_k)\cup\{0\}\right\}.$$
If $I_1=0$, set $m_1=1$. Else, define $m_1$ to be the unique element
of $\{1,\dots,k\}$ so that $I_1\in Y_{m_1}\triangle Z_{m_1}$. Assume
we have defined $I_1,\dots,I_i$ and $m_1,\dots,m_i$ for some
$i\in\{1,\dots,k-1\}$ with $I_r\in (Y_{m_r}\triangle
Z_{m_r})\cup\{0\}$ for $r=1,\dots,i$. Then set
$$I_{i+1}=\max\left\{\mathcal{U}\left(\{Y_j\triangle Z_j:j\in\{1,\dots,k\}\setminus\{m_1,\dots,m_i\}\}\right)\cup\{0\}\right\}.$$
If $I_{i+1}=0$, set
$m_{i+1}=\min\{\{1,\dots,k\}\setminus\{m_1,\dots,m_i\}\}.$
Otherwise, define $m_{i+1}$ to be the unique element of
$\{1,\dots,k\}\setminus\{m_1,\dots,m_i\}$ such that $I_{i+1}\in
Y_{m_{i+1}}\triangle Z_{m_{i+1}}$. This completes the inductive
step. Let $\{1\le j\le k:I_j>0\}=\{j_1,\dots,j_n\}$, where
$j_1<\cdots<j_n$ and put $\mathscr{J}=\{m_{j_r}:1\le r\le n\}$.
Notice that, by construction, we have that
$\{m_1,\dots,m_k\}=\{1,\dots,k\}$ and $I_{j_r}\neq I_{j_s}$ for
$1\le r<s\le n$.

Fix the primes $p_i$ for
$i\in\mathcal{I}=\{1,\dots,B\}\setminus\{I_{j_1},\dots,I_{j_n}\}$.
By the definition of the indices $I_1,\dots,I_k$, for every
$r\in\{1,\dots,n\}$ the prime number $p_{I_{j_r}}$ appears in
\eqref{s3e10}\;for $j=m_{j_r}$, but does not appear in
\eqref{s3e10}\;for $j\in\{m_{j_{r+1}},\dots,m_{j_n}\}$. So
\eqref{s3e10}, $j\in\mathscr{J}$, is a linear system with respect to
$\log p_{I_{j_1}},\dots,\log p_{I_{j_n}}$ corresponding to a
triangular matrix (up to a permutation of its rows) and a
straightforward manipulation of its rows implies that
$p_{I_{j_r}}\in[V_r,4^kV_r]$, $1\le r\le n$, for some numbers $V_r$
that depend only on the primes $p_i$ for $i\in\mathcal{I}$ and the
$k$-tuples $(Y_1,\dots,Y_k)$ and $(Z_1,\dots,Z_k)$, which we have
fixed. Therefore
$$\sum_{\substack{p_{I_{j_1}},\dots,p_{I_{j_n}}\\\eqref{s3e5},\eqref{s3e10a}}}\frac1{p_{I_{j_1}}\cdots
p_{I_{j_n}}}\le\prod_{r=1}^n\sum_{\substack{V_r\le p_{I_{j_r}}\le
4^kV_r\\p_{I_{j_r}}\in D_{E_{\bv
b}(I_{j_r})}}}\frac1{p_{I_{j_r}}}\ll_k\prod_{r=1}^n\frac1{\log(\max\{V_r,\lambda_{E_{\bv
b}(I_{j_r})-1}\})}\ll_k\prod_{r=1}^n\rho^{-E_{\bv b}(I_{j_r})}$$ and
consequently
$$\sum_{\substack{p_1,\dots,p_B\\\eqref{s3e5},\eqref{s3e10a}}}\frac1{p_1\cdots
p_B}\ll_k\prod_{r=1}^n\rho^{-E_{\bv
b}(I_{j_r})}\sum_{\substack{p_i,\;i\in\mathcal{I}\\\eqref{s3e5}}}\prod_{i\in\mathcal{I}}\frac1p_i\le(\log\rho)^{B-n}\prod_{j=1}^k\rho^{-E_{\bv
b}(I_j)},$$ by \eqref{s3e3}. Inserting the above estimate into
\eqref{s3e7}\;we deduce that
\be\label{s3e12b}\sum_{a\in\mathcal{A}(\bv
b)}\frac{W_{k+1}^P(a)}a\ll_k\frac{(\log\rho)^B}{b_1!\cdots
b_H!}\sum_{(Y_1,\dots,Y_k)\in\mathscr{P}_B}
\biggl(\sum_{(Z_1,\dots,Z_k)\in\mathscr{P}_B}\prod_{j=1}^k\rho^{-E_{\bv
b}(I_j)}\biggr)^{P-1}.\ee Next, observe that the definition of
$I_1,\dots,I_k$ implies that
$$(I_j,B]\cap\mathcal{U}\left(\{Y_{m_r}\triangle Z_{m_r}:j\le r\le
k\}\right)=\emptyset\quad(1\le j\le k)$$ or, equivalently,
$$\bigcup_{r=j}^k\left(Z_{m_r}\cap(I_j,B]\right)=\bigcup_{r=j}^k\left(Y_{m_r}\cap(I_j,B]\right)\quad(1\le j\le
k),$$ by Remark \ref{rk1}. Hence for fixed
$(Y_1,\dots,Y_k)\in\mathscr{P}_B$, $I_1,\dots,I_k\in\{0,1,\dots,B\}$
and $\bv m=(m_1,\dots,m_k)$ with $\{m_1,\dots,m_k\}=\{1,\dots,k\}$,
the number of admissible $k$-tuples
$(Z_1,\dots,Z_k)\in\mathscr{P}_B$ is at most $M_B(\bv Y_{\bv m};\bv
I)$, where $\bv Y_{\bv m}=(Y_{m_1},\dots,Y_{m_k})$, which together
with \eqref{s3e12b}\;yields that
$$\sum_{a\in\mathcal{A}(\bv
b)}\frac{W_{k+1}^P(a)}a\ll_k\frac{(\log\rho)^B}{b_1!\cdots
b_H!}\sum_{\bv Y\in\mathscr{P}_B}\left(\sum_{0\le I_1,\dots,I_k\le
B}\sum_{\bv m}M_B(\bv Y_{\bv m};\bv I)\prod_{j=1}^k\rho^{-E_{\bv
b}(I_j)}\right)^{P-1}.$$ So, by the inequality inequality
$(a+b)^{P-1}\le a^{P-1}+b^{P-1}$ for $a\ge0$ and $b\ge 0$, which
holds precisely when $1<P\le 2$, we find that
$$\sum_{a\in\mathcal{A}(\bv
b)}\frac{W_{k+1}^P(a)}a\ll_k\frac{(\log\rho)^B}{b_1!\cdots
b_H!}\sum_{\bv m}\sum_{0\le I_1,\dots,I_k\le
B}\prod_{j=1}^k(\rho^{P-1})^{-E_{\bv b}(I_j)}\sum_{\bv
Y\in\mathscr{P}_B}\left(M_B(\bv Y_{\bv m};\bv I)\right)^{P-1}.$$
Finally, note that $$\sum_{\bv Y\in\mathscr{P}_B}\left(M_B(\bv
Y_{\bv m};\bv I)\right)^{P-1}=\sum_{\bv
Y\in\mathscr{P}_B}\left(M_B(\bv Y;\bv I)\right)^{P-1}$$ for every
$\bv m=(m_1,\dots,m_k)$ with $\{m_1,\dots,m_k\}=\{1,\dots,k\}$,
which completes the proof of the lemma.
\end{proof}

\begin{lemma}\label{s3l6b} Let $P\in(1,+\infty)$ and $0\le I_1,\dots,I_k\le B$ so that
$I_{\sigma(1)}\le\cdots\le I_{\sigma(k)}$ for some permutation
$\sigma\in S_k$. Then $$\sum_{\bv Y\in\mathscr{P}_B}\left(M_B(\bv
Y;\bv
I)\right)^{P-1}\le(k+1)^B\prod_{j=1}^k\left(\frac{j-1+(k-j+2)^P}{j+(k-j+1)^P}\right)^{I_{\sigma(j)}}.$$
\end{lemma}

\begin{proof} First, we calculate $M_B(\bv Y;\bv I)$ for fixed $\bv Y=(Y_1,\dots,Y_k)\in\mathscr{P}_B$.
Set $I_0=0$, $I_{k+1}=B$, $\sigma(0)=0$, $\sigma(k+1)=k+1$ and
$$\mathcal{N}_j=(I_{\sigma(j)},I_{\sigma(j+1)}]\cap\{1,\dots,B\}\quad(0\le
j\le k).$$ In addition, put
$$Y_0=\{1,\dots,B\}\setminus\bigcup_{j=1}^kY_j$$ as well as
\be\label{s3eyij}Y_{i,j}=Y_j\cap\mathcal{N}_i\quad\text{and}\quad
y_{i,j}=|Y_{i,j}|\quad(0\le i\le k,\;0\le j\le k).\ee A $k$-tuple
$(Z_1,\dots,Z_k)\in\mathscr{P}_B$ is counted by $M_B(\bv Y;\bv I)$
if, and only if,
\be\label{s3e13a}\bigcup_{r=j}^k\left(Z_r\cap(I_j,B]\right)
=\bigcup_{r=j}^k\left(Y_r\cap(I_j,B]\right)\quad(1\le j\le k).\ee If
we set
$$Z_0=\{1,\dots,B\}\setminus\bigcup_{j=1}^kZ_j$$ and
$$Z_{i,j}=Z_j\cap\mathcal{N}_i\quad(0\le i\le k,\;0\le j\le k),$$
then \eqref{s3e13a}\;is equivalent to
\be\label{s3e13}\bigcup_{r=\sigma(j)}^kZ_{i,r}=\bigcup_{r=\sigma(j)}^kY_{i,r}\quad(0\le
i\le k,\;0\le j\le i).\ee For every $i\in\{0,1,\dots,k\}$ let
$$\chi_i:\{0,1,\dots,i+1\}\to\{\sigma(0),\sigma(1),\dots,\sigma(i),\sigma(k+1)\}$$
be the bijection uniquely determined by the property that
$\chi_i(0)<\cdots<\chi_i(i+1)$. So the sequence
$\chi_i(0),\dots,\chi_i(i+1)$ is the sequence
$\sigma(0),\sigma(1),\dots,\sigma(i),\sigma(k+1)$ ordered
increasingly. In particular, $\chi_i(0)=\sigma(0)=0$ and
$\chi_i(i+1)=\sigma(k+1)=k+1$. With this notation
\eqref{s3e13}\;becomes
$$\bigcup_{r=\chi_i(j)}^kZ_{i,r}=\bigcup_{r=\chi_i(j)}^kY_{i,r}\quad(0\le i\le k,\;0\le
j\le i),$$ which is equivalent to
$$\bigcup_{r=\chi_i(j)}^{\chi_i(j+1)-1}Z_{i,r}=\bigcup_{r=\chi_i(j)}^{\chi_i(j+1)-1}Y_{i,r}\quad(0\le
i\le k,\;0\le j\le i).$$ For each $i\in\{0,1,\dots,k\}$ let $M_i$
denote the total number of mutually disjoint $(k+1)$-tuples
$(Z_{i,0},Z_{i,1},\dots,Z_{i,k})$ such that
$$\bigcup_{r=\chi_i(j)}^{\chi_i(j+1)-1}Z_{i,r}=\bigcup_{r=\chi_i(j)}^{\chi_i(j+1)-1}Y_{i,r}\quad(0\le
j\le i).$$ Then \be\label{s3e14}M_B(\bv Y;\bv
I)=\prod_{i=0}^kM_i.\ee Moreover, it is immediate from the
definition of $M_i$ that
$$M_i=\prod_{j=0}^i(\chi_i(j+1)-\chi_i(j))^{y_{i,\chi_i(j)}+\cdots+y_{i,\chi_i(j+1)-1}}.$$
Set $v_{i,j+1}=\chi_i(j+1)-\chi_i(j)$ for $j\in\{0,\dots,i\}$. Note
that $v_{i,1}+\cdots+v_{i,i+1}=k+1$ and that $v_{i,j+1}\ge1$ for all
$j\in\{0,\dots,i\}$. Let
\be\label{s3e14b}W_{i,j}=\bigcup_{r=\chi_i(j)}^{\chi_i(j+1)-1}Y_{i,r},\quad
w_{i,j}=|W_{i,j}|\quad(0\le j\le i).\ee With this notation we have
that \be\label{s3e15}M_i=\prod_{j=0}^iv_{i,j+1}^{w_{i,j}}\quad(0\le
i\le k).\ee  Inserting \eqref{s3e15}\;into \eqref{s3e14}\;we deduce
that \be\label{s3e16}M_B(\bv Y;\bv
I)=\prod_{i=0}^k\prod_{j=0}^iv_{i,j+1}^{w_{i,j}}.\ee Therefore
$$S:=\sum_{\bv Y}\left(M_B(\bv Y;\bv
I)\right)^{P-1}=\prod_{i=0}^k\sum_{Y_{i,0},\dots,Y_{i,k}}\prod_{j=0}^i(v_{i,j+1}^{P-1})^{w_{i,j}},$$
where the sets $Y_{i,j}$ are defined by \eqref{s3eyij}. Next, we
calculate $S$. Fix $i\in\{1,\dots,k\}$. Given
$W_{i,0},\dots,W_{i,i}$, a partition of $\mathcal{N}_i$, the number
of $Y_{i,0},\dots,Y_{i,k}$ satisfying \eqref{s3e14b}\;is
$$\prod_{j=0}^i(\chi_i(j+1)-\chi_i(j))^{|W_{i,j}|}=\prod_{j=0}^iv_{i,j+1}^{w_{i,j}}.$$
Hence
$$\sum_{Y_{i,0},\dots,Y_{i,k}}\prod_{j=0}^i(v_{i,j+1}^{P-1})^{w_{i,j}}=\sum_{W_{i,0},\dots,W_{i,i}}\prod_{j=0}^i(v_{i,j+1})^{Pw_{i,j}}
=(v_{i,1}^P+\cdots+v_{i,i+1}^P)^{|\mathcal{N}_i|},$$ by the
multinomial theorem. So
$$S=\prod_{i=0}^k(v_{i,1}^P+\cdots+v_{i,i+1}^P)^{|\mathcal{N}_i|}
=\prod_{i=0}^k(v_{i,1}^P+\cdots+v_{i,i+1}^P)^{I_{\sigma(i+1)}-I_{\sigma(i)}}.$$
Finally, recall that $v_{i,1}+\cdots+v_{i,i+1}=k+1$ as well as
$v_{i,j+1}\ge1$ for all $0\le j\le i\le k$, and note that
$$\max\biggl\{\sum_{j=1}^{i+1}x_j^P:\sum_{j=1}^{i+1}x_j=k+1,x_j\ge
1\;(1\le j\le i+1)\biggr\}=i+(k+1-i)^P,$$ since the maximum of a
convex function in a simplex occurs at its vertices. Hence we
conclude that
$$S\le\prod_{i=0}^k\left(i+(k+1-i)^P\right)^{I_{\sigma(i+1)}-I_{\sigma(i)}}
=(k+1)^B\prod_{i=1}^k\left(\frac{i-1+(k-i+2)^P}{i+(k-i+1)^P}\right)^{I_{\sigma(i)}},$$
which completes the proof of the lemma.
\end{proof}

Set
$$P=\min\left\{2,\frac{(k+1)^2(\log\rho)^2}{(k+1)^2(\log\rho)^2-1}\right\}.$$
Since $(k+1)\log\rho=\frac{k+1}k\log(k+1)>1$, we have that $1<P\le
2$ and thus Lemmas \ref{s3l6a}\;and \ref{s3l6b}\;can be applied.
Moreover, for our choice of $P$ the following crucial inequality
holds.

\begin{lemma}\label{s3l5} Let $k\ge1$ and $P$ defined as above. Then
$$\frac{i-1+(k-i+2)^P}{k+1}<(\rho^{P-1})^{k-i+1}\quad(2\le i\le k).$$
\end{lemma}

\begin{proof} Set
$$f(x)=(k+1)(\rho^{P-1})^x+x-(x+1)^P-k,\;\;\;x\in[0,k].$$ It
suffices to show that $f(x)>0$ for $1\le x\le k-1$. Observe that
$f(0)=f(k)=0$. Moreover, since $1<P\le 2$,
$f^{\prime\prime\prime}(x)>0$ for all $x$. Hence $f^{\prime\prime}$
is strictly increasing. Note that
$$f^{\prime\prime}(k)=(P-1)^2(\log\rho)^2(k+1)^P-P(P-1)(k+1)^{P-2}\le0,$$
by our choice of $P$. Hence $f^{\prime\prime}(x)<0$ for $x\in(0,k)$,
that is $f$ is a concave function and thus it is positive for
$x\in(0,k)$.
\end{proof}

Let $\mathcal{B}$ be the set of vectors $(b_1,\dots,b_H)$ such that
$B_i\le i$ for all $i\in\{1,\dots,H\}$. Moreover, set
$$\lambda=\frac{(k+1)^P}{k^P+1}>1.$$

\begin{lemma}\label{s3l7} Let $k\ge1$ and $\bv b=(b_1,\dots,b_H)\in\mathcal{B}$. Then
$$\sum_{a\in\mathcal{A}(\bv b)}\frac{W_{k+1}^P(a)}a\ll_k\frac{((k+1)\log\rho)^B}{b_1!\cdots
b_H!}\left(1+\sum_{m=1}^H\lambda^{B_m-m}\right).$$
\end{lemma}

\begin{proof} Set $$t_i=\frac{i-1+(k-i+2)^P}{k+1}\quad(1\le i\le k+1).$$
Then Lemmas \ref{s3l6a}\;and \ref{s3l6b}\;imply that
\be\label{s3e17}\begin{split}\sum_{a\in\mathcal{A}(\bv
b)}\frac{W_{k+1}^P(a)}a&\ll_k\frac{((k+1)\log\rho)^B}{b_1!\cdots
b_H!}\sum_{0\le I_1\le\cdots\le I_k\le
B}\prod_{j=1}^k(\rho^{P-1})^{-E_{\bv b}(I_j)}\left(\frac{t_j}{t_{j+1}}\right)^{I_j}\\
&\ll_k\frac{((k+1)\log\rho)^B}{b_1!\cdots b_H!}\sum_{0=j_0\le
j_1\le\cdots\le j_k\le
H}(\rho^{P-1})^{-(j_1+\cdots+j_k)}\\
&\qquad\qquad\qquad\qquad\qquad\quad\qquad\times\prod_{i=1}^k\sum_{B_{j_{i-1}}\le
I_i\le
B_{j_i}}\left(\frac{t_i}{t_{i+1}}\right)^{I_i}\\
&\ll_k\frac{((k+1)\log\rho)^B}{b_1!\cdots b_H!}\sum_{0\le
j_1\le\cdots\le j_k\le
H}\prod_{i=1}^k(\rho^{P-1})^{-j_i}\left(\frac{t_i}{t_{i+1}}\right)^{B_{j_i}},\end{split}\ee
since $t_1>\dots>t_k>t_{k+1}=1$. Moreover,
$$\prod_{i=1}^k\left(\frac{t_i}{t_{i+1}}\right)^{B_{j_i}}
\le\left(\frac{t_1}{t_2}\right)^{B_{j_1}}\prod_{i=2}^k\left(\frac{t_i}{t_{i+1}}\right)^{j_i}
=\left(\frac{t_1}{t_2}\right)^{B_{j_1}}t_2^{j_1}\prod_{i=2}^kt_i^{j_i-j_{i-1}},$$
since $\bv b\in\mathcal{B}$. Thus, by setting $r_1=j_1$ and
$r_i=j_i-j_{i-1}$ for $i=2,\dots,k$, we deduce that
\be\begin{split}\prod_{i=1}^k(\rho^{P-1})^{-j_i}\left(\frac{t_i}{t_{i+1}}\right)^{B_{j_i}}&\le
(\rho^{P-1})^{-(j_1+\cdots+j_k)}\left(\frac{t_1}{t_2}\right)^{B_{j_1}}t_2^{j_1}\prod_{i=2}^kt_i^{j_i-j_{i-1}}\\
&=\left(\frac{t_1}{t_2}\right)^{B_{r_1}}\left(\frac{t_2}{\rho^{(P-1)k}}\right)^{r_1}
\prod_{i=2}^k\left(\frac{t_i}{(\rho^{P-1})^{k-i+1}}\right)^{r_i}\\
&=\lambda^{B_{r_1}-r_1}\prod_{i=2}^k\left(\frac{t_i}{(\rho^{P-1})^{k-i+1}}\right)^{r_i},\nonumber\end{split}\ee
since $\rho^{(P-1)k}=t_1$ and $t_1/t_2=\lambda$. Consequently,
\be\label{s3e17a}\begin{split}\sum_{0\le j_1\le\cdots\le j_k\le
H}\prod_{i=1}^k(\rho^{P-1})^{-j_i}\left(\frac{t_i}{t_{i+1}}\right)^{B_{j_i}}&\le\sum_{\substack{r_1+\cdots+r_k\le
H\\r_i\ge 0\;(1\le i\le
k)}}\lambda^{B_{r_1}-r_1}\prod_{i=2}^k\left(\frac{t_i}{(\rho^{P-1})^{k-i+1}}\right)^{r_i}\\
&\le\sum_{\substack{0\le r_i\le
H\\1\le i\le k}}\lambda^{B_{r_1}-r_1}\prod_{i=2}^k\left(\frac{t_i}{(\rho^{P-1})^{k-i+1}}\right)^{r_i}\\
&\ll_k\sum_{r_1=0}^H\lambda^{B_{r_1}-r_1},\end{split}\ee since
$t_i<(\rho^{P-1})^{k-i+1}$ for $i=2,\dots,k$ by Lemma \ref{s3l5}.
Inserting \eqref{s3e17a}\;into \eqref{s3e17}\;completes the proof of
the lemma.
\end{proof}

We will now use Lemmas \ref{s3l2}, \ref{s3l4}\;and \ref{s3l7}\;to
bound $\widetilde{H}^{(k+1)}(x,\bv y,2\bv y)$ from below. Recall,
from the beginning of this section, that we have assumed that
$y_1>C_1$ for a sufficiently large constant $C_1$.

\begin{lemma}\label{s3l8} Let $k\ge 1$, $0<\delta<1$ and $c\ge 1$. Consider $x\ge 1$
and $3\le y_1\le y_2\le\cdots\le y_k\le y_1^c$ with
$2^{k+1}y_1\cdots y_k\le x/y_1^\delta$. For a positive integer
$N=N(k)$ set $$H=\left\lfloor\frac{\log\log
y_1}{\log\rho}-L_k\right\rfloor\quad\text{and}\quad B=H-N+1.$$ If $N$
is large enough, then $$H^{(k+1)}(x,\bv y,2\bv y)
\gg_{k,\delta,c}\frac{x}{(\log y_1)^{k+1}}(B(k+1)\log\rho)^B{\rm
Vol}(Y_B(N)),$$ where $Y_B(N)$ is the set of
$\bv\xi=(\xi_1,\dots,\xi_B)\in\SR^B$ satisfying \medskip
\begin{enumerate}\item $0\le\xi_1\le\cdots\le\xi_B\le 1$;
\medskip
\item $\xi_{i+1}\ge i/B$ $(1\le i\le B-1)$;
\medskip
\item $\sum_{j=1}^B\lambda^{j-B\xi_j}\le\lambda^N.$
\end{enumerate}
\end{lemma}

\begin{proof} Let $\mathcal{B}^*$ be the set of vectors $(b_1,\dots,b_H)\in(\SN\cup\{0\})^H$ such that $b_i=0$ for $i<N$,
\be\label{s3e20a}B_i\le i-N+1\quad(N\le i\le H)\ee and
\be\label{s3e20b}\sum_{m=N}^H\lambda^{B_m-m}\le\frac{\lambda+\lambda^{-N}}{1-1/\lambda}.\ee
Lemma \ref{s3l3}\;and the definition of $H$ imply that
$\log\lambda_H\le\rho^{H+L_k}\le\log y_1.$ Hence
\be\label{s3e17b}\bigcup_{\bv b\in\mathcal{B}^*}\mathcal{A}(\bv
b)\subset\{a\in\SN:P^+(a)\le y_1,\;\mu^2(a)=1\}.\ee Fix for the
moment $\bv b\in\mathcal{B}^*\subset\mathcal{B}$. By Lemma
\ref{s3l7}\;and relation \eqref{s3e20b}\;we have that
\be\label{s3e22}\sum_{a\in\mathcal{A}(\bv
b)}\frac{W_{k+1}^P(a)}a\ll_k\frac{((k+1)\log\rho)^B}{b_N!\cdots
b_H!}\biggl(1+\sum_{m=N}^H\lambda^{B_m-m}\biggr)\ll_k\frac{((k+1)\log\rho)^B}{b_N!\cdots
b_H!}.\ee Also, if $N$ is large enough, then Lemma \ref{s3l3}\;and
relation \eqref{s3e20a}\;imply that
\be\label{s3e21}\begin{split}\sum_{a\in\mathcal{A}(\bv
b)}\frac{\tau_{k+1}(a)}a&=(k+1)^B\prod_{j=N}^H\frac1{b_j!}\biggl(\sum_{p_1\in
D_j}\frac1{p_1}\sum_{\substack{p_2\in D_j\\p_2\neq
p_1}}\frac1{p_2}\cdots\sum_{\substack{p_{b_j}\in
D_j\\p_{b_j}\notin\{p_1,...,p_{b_j-1}\}}}\frac1{p_{b_j}}\biggl)\\
&\ge\frac{(k+1)^B}{b_N!\cdots b_H!}\prod_{j=N}^H\biggl(\log\rho-\frac{b_j}{\lambda_{j-1}}\biggr)^{b_j}\\
&\ge\frac{((k+1)\log\rho)^B}{b_N!\cdots b_H!}\prod_{j=N}^H\biggl(1-\frac{j-N+1}{(\log\rho)\exp\{\rho^{j-L_k-1}\}}\biggr)^{j-N+1}\\
&\ge\frac12\frac{((k+1)\log\rho)^B}{b_N!\cdots b_H!}\end{split}\ee
as well as
\be\label{s3e21b}\frac{\phi(a)}a\ge\prod_{j=N}^H\left(1-\frac1{\lambda_{j-1}}\right)^{b_j}
\ge\prod_{j=N}^H\left(1-\frac1{\exp\{\rho^{j-L_k-1}\}}\right)^{j-N+1}
\ge\frac12\ee for $a\in\mathcal{A}(\bv b)$. Combining Lemma
\ref{s3l4}\;with relations \eqref{s3e22}\;and \eqref{s3e21}\;we
deduce that $$\sum_{a\in\mathcal{A}(\bv
b)}\frac{L^{(k+1)}(a)}a\gg_k\frac{((k+1)\log\rho)^B}{b_N!\cdots
b_H!}.$$ The above relation together with \eqref{s3e17b},
\eqref{s3e21b}\;and Lemma \ref{s3l2}\;yields that
$$\widetilde{H}^{(k+1)}(x,\bv y,2\bv
y)\gg_{k,\delta,c}x\frac{((k+1)\log\rho)^B}{(\log
y_1)^{k+1}}\sum_{\bv b\in\mathcal{B}^*}\frac{1}{b_N!\cdots b_H!}.$$
For $i\in\{1,\dots,B\}$ set $g_i=b_{N-1+i}$ and let
$G_i=g_1+\cdots+g_i$. Then \be\label{s3e24}G_i=B_{i+N-1}\le
i\quad(1\le i\le B)\ee and
\be\label{s3e26}\sum_{i=1}^B\lambda^{G_i-i}=\lambda^{N-1}\sum_{m=N}^H\lambda^{B_m-m}\le\frac{\lambda^N+1/\lambda}{1-1/\lambda},\ee
by \eqref{s3e20a}\;and \eqref{s3e20b}, respectively. With this
notation we have that \be\label{s3e27}\widetilde{H}^{(k+1)}(x,\bv
y,2\bv y)\gg_{k,\delta,c}x\frac{((k+1)\log\rho)^B}{(\log
y_1)^{k+1}}\sum_{\bv g\in\mathcal{G}}\frac{1}{g_1!\cdots g_B!},\ee
where $\mathcal{G}$ is the set of vectors $\bv g=(g_1,\dots,g_B)$ of
non-negative integers $g_1,\dots,g_B$ with $g_1+\cdots+g_B=B$ and
such that \eqref{s3e24} and \eqref{s3e26}\;hold. For such a $\bv g$
let $R(\bv g)$ be the set of $\bv x\in\SR^B$ such that $0\le
x_1\le\cdots\le x_B\le B$ and exactly $g_i$ of the $x_j$'s lie in
$[i-1,i)$ for each $i$. Then \be\label{s3e28}\sum_{\bv
g\in\mathcal{G}}\frac{1}{g_1!\cdots g_B!}=\sum_{\bv
g\in\mathcal{G}}\vol(R(\bv g))=\vol(\cup_{\bv g\in\mathcal{G}}R(\bv
g)).\ee We claim that \be\label{s3e29}\vol(\cup_{\bv
g\in\mathcal{G}}R(\bv g))\ge B^B\vol(Y_B(N)).\ee Take $\bv\xi\in
Y_B(N)$ with $\xi_B<1$ and set $x_j=B\xi_j$. Let $g_i$ be the number
of $x_j$'s lying in $[i-1,i)$. It suffices to show that $\bv
g=(g_1,\dots,g_B)\in\mathcal{G}$. Condition (2) in the definition of
$Y_B(N)$ implies that
$$x_{i+1}\ge i\quad(1\le i\le B-1),$$ which yields
\eqref{s3e24}. Finally, condition (3) in the definition of $Y_B(N)$
gives us that
\be\begin{split}\frac{\lambda^N}{1-1/\lambda}&\ge\frac1{1-1/\lambda}\sum_{j=1}^B\lambda^{j-x_j}
\ge\frac1{1-1/\lambda}\sum_{i=1}^B\lambda^{-i}\sum_{j:x_j\in[i-1,i)}\lambda^j\ge\sum_{i=1}^B\sum_{m=i}^B\lambda^{-m}\sum_{j:x_j\in[i-1,i)}\lambda^j\\
&=\sum_{m=1}^B\lambda^{-m}\sum_{j:x_j<m}\lambda^j\ge\sum_{\substack{1\le
m\le
B\\G_m>0}}\lambda^{-m+G_m}\ge-\frac1{\lambda-1}+\sum_{m=1}^B\lambda^{-m+G_m},\nonumber\end{split}\ee
that is \eqref{s3e26}\;holds. To conclude, we have showed that $\bv
g\in\mathcal{G}$, which proves that inequality \eqref{s3e29}\;does
hold. This fact along with \eqref{s3e27}\;and
\eqref{s3e28}\;completes the proof of the lemma.
\end{proof}

Next, we give a lower bound to the volume of $Y_B(N)$.

\begin{lemma}\label{s3l9} Suppose that $N$ is large enough. Then
$$\vol(Y_B(N))\gg\frac1{(B+1)!}.$$
\end{lemma}

The proof of the above lemma will be given in Section 5. If we use
Lemmas \ref{s3l8}\;and \ref{s3l9}, we get that
\be\begin{split}\widetilde{H}^{(k+1)}(x,\bv y,2\bv
y)\gg_{k,\delta,c}\frac{x}{(\log
y_1)^{k+1}}\frac{(B(k+1)\log\rho)^B}{B\cdot B!}&\asymp\frac{x}{(\log
y_1)^{k+1}}\frac{(e(k+1)\log\rho)^B}{B^{3/2}}\\
&\asymp_{k,\delta}\frac{x}{(\log y_1)^{Q(\frac1{\log\rho})}(\log\log
y_1)^{3/2}},\nonumber\end{split}\ee which completes the proof of
Theorem \ref{th2}\;and thus the proof of the lower bound in Theorem
\ref{th1}.


\section{Upper bounds.}\label{ub}

The proof of the upper bound we will give follows the corresponding
arguments in \cite{kf1}. The argument is simplified slightly by
Lemma \ref{s2l4}. As in the proof of the lower bounds, we will
assume that $y_1>C_2$ for some large enough positive constant
$C_2=C_2(k,\delta)$; else, we may use the trivial bound
$H^{(k+1)}(x,\bv y,2\bv y)\le x$ and immediately get the upper bound
in Theorem \ref{th1}.

\medskip

For $\bv y,\bv z\in\SR^k$ and $x\ge1$ define
$$H^{(k+1)}_*(x,\bv y,\bv z)=\lvert\{n\le x:\mu^2(n)=1,\;\tau_{k+1}(n,\bv y,\bv z)\ge 1\}\rvert.$$ Also, for $t\ge 1$
set
$$S^{(k+1)}(t)=\sum_{\substack{P^+(a)\le t\\\mu^2(a)=1}}\frac{L^{(k+1)}(a)}a.$$
Then we have the following estimate.
\begin{lemma}\label{s4l1} Let $3\le y_1,\dots,y_k\le x$ with
$y_1\cdots y_k\le x/2^{k+1}$. Set $z_i=2y_i$ for $i=1,\dots,k$ and
$z_{k+1}=\frac x{y_1\cdots y_k}$. Then
$$H^{(k+1)}_*(x,\bv y,2\bv y)-H^{(k+1)}_*(x/2,\bv y,2\bv
y)\ll_kx(\log Z)^{k+1}\sum_{i=1}^{k+1}\frac{S^{(k+1)}(z_i)}{(\log
z_i)^{2k+2}},$$ where $Z=\max_{1\le j\le k}z_j$.
\end{lemma}

\begin{proof} Let $n\in(x/2,x]$ be a square-free integer such that
$\tau_{k+1}(n,\bv y,2\bv y)\ge 1$. Then we may write $n=d_1\cdots
d_{k+1}$ with $y_i<d_i\le2y_i$ for $i=1,\dots,k$. Hence
$d_{k+1}\in(\frac x{2^{k+1}y_1\cdots y_k},\frac x{y_1\cdots y_k}].$
So if we set $y_{k+1}=\frac x{2^{k+1}y_1\cdots y_k}\ge1$, then
$y_i<d_i\le z_i$ for $i=1,\dots,k+1$. For a unique permutation
$\sigma\in S_{k+1}$ we have that
$P^+(d_{\sigma(1)})<\cdots<P^+(d_{\sigma(k+1)}).$ Set
$p_j=P^+(d_{\sigma(j)})$ for $j=1,\dots,k+1$. Then we may write
$n=aa^\prime p_1\cdots p_kb$, where $P^+(a)<p_1<p_k<P^-(b)$ and all
the prime divisors of $a^\prime$ lie in $(p_1,p_k)$. Observe that
$d_{\sigma(1)}=p_1d$ for some integer $d$ with $P^+(d)<p_1$. In
particular, $d|a$ and thus $y_{\sigma(1)}<d_{\sigma(1)}=p_1d\le
p_1a$. Consequently,
$$p_1>Q=\max\left\{P^+(a),\frac{y_{\sigma(1)}}a\right\}.$$ Moreover,
$$P^+(a^\prime)\le p_k=\max_{1\le j\le k}P^+(d_{\sigma(j)})\le\max_{1\le i\le k}P^+(d_i)\le Z,$$ by the choice of $\sigma$. In particular,
$a^\prime\in\mathscr{P}(Q,Z)$. Also, we have that $b>p_k$, since
$p_{k+1}|b$ and $p_{k+1}>p_k$. Next, note that
$$(d_{\sigma(1)}/p_1)\cdots(d_{\sigma(k)}/p_k)|aa^\prime\quad\text{and}\quad\frac{y_{\sigma(i)}}{p_i}<\frac{d_{\sigma(i)}}{p_i}\le
\frac{z_{\sigma(i)}}{p_i}\quad(1\le i\le k).$$ So there are numbers
$c_1,\dots,c_k\in\{1,2,2^2,\dots,2^k\}$ such that
$$\bv y^\prime:=\left(\log\frac{c_1y_{\sigma(1)}}{p_1},\dots,\log\frac{c_ky_{\sigma(k)}}{p_k}\right)\in\mathcal{L}^{(k+1)}(aa^\prime).$$
Hence \be\label{s4e0}\begin{split}&H^{(k+1)}_*(x,\bv z,2\bv
z)-H^{(k+1)}_*(x/2,\bv z,2\bv z)\\
&\quad\le\sum_{\substack{\sigma\in
S_{k+1}\\c_1,\dots,c_k}}\sum_{P^+(a)\le
z_{\sigma(1)}}\sum_{\substack{a^\prime\in\mathscr{P}(Q,Z)\\
\mu^2(aa^\prime)=1}}\sum_{\substack{Q<p_1<\cdots<p_k\\\bv
y^\prime\in\mathcal{L}^{(k+1)}(aa^\prime)}}\sum_{\substack{p_k<b\le
x/(aa'p_1\cdots p_k)\\P^-(b)>p_k}}1.\end{split}\ee Note that the
innermost sum in the right hand side of \eqref{s4e0}\;is
$$\ll\frac{x}{aa^\prime p_1\cdots p_k\log p_k}\ll\frac{x}{aa^\prime
p_1\cdots p_k\log(2Q)},$$ by Lemma \ref{s2l1}. Therefore
\eqref{s4e0}\;becomes
\be\label{s4e2c}\begin{split}&H^{(k+1)}_*(x,\bv z,2\bv
z)-H^{(k+1)}_*(x/2,\bv z,2\bv z)\\
&\quad\ll x\sum_{\substack{\sigma\in
S_{k+1}\\c_1,\dots,c_k}}\sum_{P^+(a)\le
z_{\sigma(1)}}\sum_{\substack{a^\prime\in\mathscr{P}(Q,Z)\\
\mu^2(aa^\prime)=1}}\frac1{aa^\prime\log(2Q)}\sum_{\substack{Q<p_1<\cdots<p_k\\\bv
y^\prime\in\mathcal{L}^{(k+1)}(aa^\prime)}}\frac1{p_1\cdots
p_k}.\end{split}\ee Fix $a$, $a'$, $\sigma$ and $c_1,\dots,c_k$ as
above. Let $m_1\cdots m_k|aa^\prime$ and set
$$I=[\log(m_1/2),\log m_1)\times\cdots\times[\log(m_k/2),\log m_k)$$
as well as $$U_i=\frac{c_i y_{\sigma(i)}}{2m_i}\quad(1\le i\le k).$$
Then $\bv y^\prime\in3I=[\log
(m_1/4),\log(2m_1))\times\cdots\times[\log (m_k/4),\log (2m_k))$ if,
and only if, $U_i<p_i\le 8U_i$ for all $i\in\{1,\dots,k\}$. Thus
\be\label{s4e2d}\begin{split}\sum_{\substack{Q<p_1<\cdots<p_k\\\bv
y^\prime\in 3I}}\frac1{p_1\cdots
p_k}&\le\prod_{i=1}^k\sum_{\substack{U_i<p_i\le8U_i\\p_i>Q}}\frac1p_i\ll_k\prod_{i=1}^k\frac1{\log(\max\{2Q,U_i\})}\le\frac1{\log^k(2Q)}.\end{split}\ee
If $\{I_r\}_{r=1}^R$ is the collection of the cubes
$[\log(m_1/2),\log m_1)\times \cdots \times[\log(m_k/2),\log m_k)$
with $m_1\cdots m_k|aa^\prime$, then Lemma~\ref{s2l3} implies that
there exists a sub-collection $\{I_{r_s}\}_{s=1}^S$ of mutually
disjoint cubes such that
$$\mathcal{L}^{(k+1)}(aa^\prime)\subset\bigcup_{s=1}^S3I_{r_s}\quad\text{and}\quad S(\log
2)^k=\vol\left(\bigcup_{s=1}^SI_{r_s}\right)\le
L^{(k+1)}(aa^\prime).$$ Thus \eqref{s4e2d}\;along with Lemma
\ref{s3l1}(b) yield that
$$\sum_{\substack{Q<p_1<\cdots<p_k\\\bv{y^\prime}\in\mathcal{L}^{(k+1)}(aa^\prime)}}\frac1{p_1p_2\cdots
p_k}\ll_k\frac{L^{(k+1)}(aa^\prime)}{\log^k(2Q)}\le\frac{\tau_{k+1}(a^\prime)L^{(k+1)}(a)}{\log^k(2Q)}.$$
Inserting the above estimate into \eqref{s4e2c}\;we find that
\be\begin{split}&H^{(k+1)}_*(x,\bv z,2\bv z)-H^{(k+1)}_*(x/2,\bv
z,2\bv z)\\
&\quad\ll_k\sum_{\sigma\in S_{k+1}}\sum_{\substack{P^+(a)\le
z_{\sigma(1)}\\\mu^2(a)=1}}\sum_{\substack{a^\prime\in\mathscr{P}(Q,Z)\\
\mu^2(a^\prime)=1}}\frac{\tau_{k+1}(a^\prime)}{a^\prime}\frac{L^{(k+1)}(a)}{a\log^{k+1}(2Q)}\\
&\quad\ll_k(\log Z)^{k+1}\sum_{\sigma\in
S_{k+1}}\sum_{\substack{a\le
P^+(z_{\sigma(1)})\\\mu^2(a)=1}}\frac{L^{(k+1)}(a)}{a\log^{2k+2}(P^+(a)+z_{\sigma(1)}/a)},\nonumber\end{split}\ee
since $$\sum_{\substack{a^\prime\in\mathscr{P}(Q,Z)\\
\mu^2(a^\prime)=1}}\frac{\tau_{k+1}(a^\prime)}{a^\prime}=\prod_{Q<p\le
Z}\left(1+\frac{k+1}p\right)\ll_k\left(\frac{\log
Z}{\log(2Q)}\right)^{k+1}.$$ To complete the proof use Lemma
\ref{s2l4}(b)\;to see that
$$
\sum_{\substack{P^+(a)\le t\\\mu^2(a)=1}}\frac{L^{(k+1)}(a)}{a\log^{2k+2}(P^+(a)+t/a)}
	\ll_k\frac{S^{(k+1)}(t)}{(\log t)^{2k+2}}\quad(t\ge2).
$$
\end{proof}

Next, we use the lemma we just proved to bound $H^{(k+1)}(x,\bv
y,2\bv y)$ from above.

\begin{lemma}\label{s4l2} Let $k\ge2$ and $0<\delta\le1$. For $x\ge3$ and
$3\le y_1\le\cdots\le y_k$ with $2^{k+1}y_1\cdots y_k\le
x/y_1^\delta$ we have that $$H^{(k+1)}(x,\bv y,2\bv y)\ll_{k,\delta}
x\frac{(\log y_k)^{k+1}}{(\log y_1)^{2k+2}}S^{(k+1)}(y_1).$$
\end{lemma}

\begin{proof} First, we reduce the problem to estimating
$H^{(k+1)}_*(x,\bv y,2\bv y)$. Let $n\in\SN$ with $\tau_{k+1}(n,\bv
y,2\bv y)\ge1$. Write $n=n^\prime n^{\prime\prime}$ with $n^\prime$
being square-free, $n^{\prime\prime}$ square-full and
$(n^\prime,n^{\prime\prime})=1$. The number of $n\le x$ with
$n^{\prime\prime}>(\log y_1)^{2k+2}$ is
$$\le
x\sum_{\substack{n^{\prime\prime}\;{\rm
square-full}\\n^{\prime\prime}>(\log
y_1)^{2k+2}}}\frac1{n^{\prime\prime}}\ll\frac x{(\log y_1)^{k+1}}.$$
Assume now that $n^{\prime\prime}\le(\log y_1)^{2k+2}.$ For some
product $f_1\cdots f_k|n^{\prime\prime}$ there is a product
$e_1\cdots e_k|n^\prime$ such that $y_i/f_i<e_i\le2y_i/f_i$ for
$i=1,\dots,k$. Therefore \be\label{s4e4}\begin{split}H^{(k+1)}(x,\bv
y,2\bv y)\le\sum_{\substack{n^{\prime\prime}\;{\rm
square-full}\\n^{\prime\prime}\le(\log y_1)^{2k+2}}}\sum_{f_1\cdots
f_k|n^{\prime\prime}}H^{(k+1)}_*\left(\frac
x{n^{\prime\prime}},\left( \frac{y_1}{f_1},\dots,\frac{y_k}{f_k}\right),2\left(\frac{y_1}{f_1},\dots,\frac{y_k}{f_k}\right)\right)\\
	+ O\left(\frac x{(\log y_1)^{k+1}}\right).
\end{split}\ee Fix a
square-full integer $n^{\prime\prime}\le(\log y_1)^{2k+2}$ and
positive integers $f_1,\dots,f_k$ with $f_1\cdots
f_k|n^{\prime\prime}$. Put $x^\prime=x/n^{\prime\prime}$ and
$y_i^\prime=y_i/f_i$ for $i=1,\dots,k$. Each
$n^\prime\in(x^{\prime}/(\log y_1)^{k+1},x^{\prime}]$ lies in a
interval $(2^{-r-1}x^{\prime},2^{-r}x^{\prime}]$ for some integer
$0\le r\le\frac{k+1}{\log 2}\log_2y_1$. We will apply Lemma
\ref{s4l1}\;with $2^{-r}x^{\prime}$ in place of $x$ and
$y_1^\prime,\dots,y_k^\prime$ in place of $y_1,\dots,y_k$. Set
$z_i^\prime=2y_i^\prime$ for $i=1,\dots,k$ and
$z_{k+1}^\prime=2^{-r}x^\prime/(y_1^\prime\cdots y_k^\prime)$.
Moreover, let $\bv y^\prime=(y_1^\prime,\dots,y_k^\prime)$. Note
that $\sqrt{y_1}\le z_j^\prime\le2y_k$ for all $j\in\{1,\dots,k\}$
and
$$z_{k+1}^\prime=\frac{xf_1\cdots f_k}{2^rn^{\prime\prime}y_1\cdots y_k}\ge\frac x{(\log y_1)^{3k+3}y_1\cdots y_k}\ge
y_1^{\delta/2},$$ provided that $C_2$ is large enough. So
\be\begin{split}H^{(k+1)}_*(x^\prime,\bv y^\prime,2\bv
y^\prime)&\ll_k\frac{x^\prime}{(\log y_1)^{k+1}} +\sum_{0\le
r\le\frac{k+1}{\log
2}\log_2y_1}\frac{x^\prime}{2^r}\sum_{i=1}^{k+1}(\log y_k)^{k+1}
\frac{S^{(k+1)}(z^\prime_i)}{(\log z_i^\prime)^{2k+2}}\\
&\ll_{k,\delta}\frac{x^\prime}{(\log y_1)^{k+1}}+x^\prime(\log
y_k)^{k+1}\max\left\{\frac{S^{(k+1)}(t)}{(\log t)^{2k+2}}:t\ge
y_1^{\delta/2}\right\}.\nonumber\end{split}\ee By the above estimate,
\eqref{s4e4}\;and the straightforward inequalities
$$\sum_{n^{\prime\prime}\;{\rm
square-full}}\frac{\tau_{k+1}(n^{\prime\prime})}{n^{\prime\prime}}\ll_k1$$
and $S^{(k+1)}(t)\ge L^{(k+1)}(1)=(\log 2)^k$, we deduce that
$$H^{(k+1)}(x,\bv y,2\bv y)\ll_{k,\delta}x(\log
y_k)^{k+1}\max\left\{\frac{S^{(k+1)}(t)}{(\log t)^{2k+2}}:t\ge
y_1^{\delta/2}\right\}.$$ Finally, note that for every $t\ge
y_1^{\delta/2}$ we have that
$$S^{(k+1)}(t)\le\sum_{\substack{P^+(a_1)\le
y_1^{\delta/2}\\\mu^2(a_1)=1}}\frac{L^{(k+1)}(a_1)}{a_1}
\sum_{\substack{a_2\in\mathscr{P}(y_1^{\delta/2},t)\\\mu^2(a_2)=1}}\frac{\tau_{k+1}(a_2)}{a_2}\ll_{k,\delta}\left(\frac{\log
t}{\log y_1}\right)^{k+1}\sum_{\substack{P^+(a_1)\le
y_1\\\mu^2(a_1)=1}}\frac{L^{(k+1)}(a_1)}{a_1},$$ where we used Lemma
\ref{s3l1}(b). Therefore
$$\frac{S^{(k+1)}(t)}{(\log t)^{2k+2}}\ll_k\frac{S^{(k+1)}(y_1)}{(\log y_1\log
t)^{k+1}}\ll_{k,\delta}\frac{S^{(k+1)}(y_1)}{(\log y_1)^{2k+2}},$$
which completes the proof of the lemma.
\end{proof}

We proceed by bounding $S^{(k+1)}(y_1)$ from above. First, define
$$\omega_k(a)=\lvert\{p|a:p>k\}\rvert$$ and $$S^{(k+1)}_r(y_1)=\sum_{\substack{P^+(a)\le
y_1\\\omega_k(a)=r,\mu^2(a)=1}}\frac{L^{(k+1)}(a)}a.$$ Then we have
the following estimate of $S^{(k+1)}_r(y_1)$ when $1\le
r\ll_k\log\log y_1$.

\begin{lemma}\label{s4l3} Let $v=\left\lfloor\frac{\log\log y_1}{\log\rho}\right\rfloor$ and assume
that $1\le r\le (10k)v$. Then
$$S^{(k+1)}_r(y_1)\ll_k((k+1)\log\log
y_1)^rU_r(v;k),$$ where
$$U_r(v;k)=\idotsint\limits_{0\le\xi_1\le\cdots\le\xi_r\le
1}\left(\min_{0\le j\le
r}\rho^{-j}(\rho^{v\xi_1}+\cdots+\rho^{v\xi_j}+1)\right)^kd\bv\xi.$$
\end{lemma}

\begin{proof} For the sets $D_j$ constructed in Section
3 we have that $$\{p\;\text{prime}:k<p\le
y_1\}\subset\bigcup_{j=1}^{v+L_k+1}D_j,$$ by Lemma \ref{s3l3}.
Consider a square-free integer $a=bp_1\cdots p_r$ with $P^+(b)\le
k<p_1<\cdots<p_r$ and define $j_i$ by $p_i\in D_{j_i}$, $1\le i\le
r$. By Lemmas \ref{s3l1}\;and \ref{s3l3}, we have
\be\begin{split}L^{(k+1)}(a)&\le\tau_{k+1}(b)L^{(k+1)}(p_1\cdots
p_r)\\
&\le\tau_{k+1}(b)\min_{0\le s\le
r}(k+1)^{r-s}(\log p_1+\cdots+\log p_s+\log2)^k\nonumber\\
&\ll_k\tau_{k+1}(b)(k+1)^r F(\bv j),\end{split}\ee where
$$F(\bv j):=\left(\min_{0\le s\le
r}\rho^{-s}(\rho^{j_1}+\cdots+\rho^{j_s}+1)\right)^k.$$ Furthermore,
we have that
$$\sum_{\substack{P^+(b)\le k\\\mu^2(b)=1}}\frac{\tau_{k+1}(b)}b\ll_k1.$$
So if $\mathcal{J}$ denotes the set of vectors $\bv
j=(j_1,\dots,j_r)$ satisfying $1\le j_1\le\cdots\le j_r\le v+L_k+1$,
then \be\label{s4e100}S^{(k+1)}_r(y_1)\ll_k(k+1)^r\sum_{\bv
j\in\mathcal{J}}F(\bv j)\sum_{\substack{p_1<\cdots<p_r\\p_i\in
D_{j_i}\;(1\le i\le r)}}\frac1{p_1\cdots p_r}.\ee Fix $\bv
j=(j_1,\dots,j_r)\in\mathcal{J}$ and let $b_s=|\{1\le i\le
r:j_i=s\}|$ for $1\le s\le v+L_k+1$. By \eqref{s3e3}\;and the
hypothesis that $r\le 10kv$, the sum over $p_1,\dots,p_r$
in~\eqref{s4e100} is at most
\be\label{s4e101}\begin{split}\prod_{s=1}^{v+L_k+1}\frac1{b_s!}\left(\sum_{p\in
D_s}\frac1p\right)^{b_s}\le\frac{(\log\rho)^r}{b_1!\cdots
b_{v+L_k+1}!}&=((v+L_k+1)\log\rho)^r\vol(I(\bv j))\\
&\ll_k(\log\log y_1)^r\vol(I(\bv j)),\end{split}\ee where
$$I(\bv j):=\{0\le
\xi_1\le\cdots\le\xi_r\le1:j_i-1\le(v+L_k+1)\xi_i<j_i\;(1\le i\le
r)\},$$ because for each $(\xi_1,\dots,\xi_r)\in I(\bv j)$ and
$s\in\{1,\dots,v+L_k-1\}$ there are exactly $b_s$ numbers $\xi_j$
satisfying $s-1\le(v+L_k+1)\xi_i<s$ and $\vol(\{0\le x_1\le\cdots\le
x_b\le 1\})=1/b!$. Inserting \eqref{s4e101}\;into \eqref{s4e100}\;we
deduce that
$$S^{(k+1)}_r(y_1)\ll_k((k+1)\log\log y_1)^r\sum_{\bv j\in\mathcal{J}}F(\bv
j)\vol(I(\bv j)).$$ Finally, note that for every $\bv\xi\in I(\bv
j)$ we have that
$\rho^{j_i}\le\rho^{1+(v+L_k+1)\xi_i}\le\rho^{L_k+2}\rho^{v\xi_i}$
and thus
$$F(\bv j)\ll_k\left(\min_{0\le g\le
r}\rho^{-g}(\rho^{v\xi_1}+\cdots+\rho^{v\xi_g}+1)\right)^k,$$ which
in turn implies that
$$\sum_{\bv j\in\mathcal{J}}F(\bv j)\vol(I(\bv
j))\ll_kU_r(v;k).$$ This completes the proof.
\end{proof}

The proof of the next lemma will be given in section \ref{ord_stat}.

\begin{lemma}\label{s4l4} Suppose $r,v$ are integers satisfying $1\le r\le
(10k)v$. Then
$$U_r(v;k)\ll_k\frac{1+\abs{v-r}}{(r+1)!((k+1)^{r-v}+1)}.$$
\end{lemma}

We combine Lemmas \ref{s4l3}\;and \ref{s4l4}\;to estimate
$S^{(k+1)}(y_1)$.

\begin{lemma}\label{s4l5} We have that
$$S^{(k+1)}(y_1)\ll_k\frac{(\log
y_1)^{k+1-Q(\frac1{\log\rho})}}{(\log\log y_1)^{3/2}}.$$
\end{lemma}

\begin{proof} Let $v=\left\lfloor\frac{\log\log
y_1}{\log\rho}\right\rfloor$. By Lemmas \ref{s4l3}\;and
\ref{s4l4}\;we have that \be\label{s4e7}\sum_{v<r\le
10kv}S_r^{(k+1)}(y_1)\ll_k\sum_{v<r\le
10kv}\frac{(r-v)}{(k+1)^{r-v}}\frac{((k+1)\log\log
y_1)^r}{(r+1)!}\ll_k\frac{((k+1)\log\log y_1)^v}{(v+1)!},\ee since
$\log\rho<1$, and \be\label{s4e8}\sum_{1\le r\le
v}S_r^{(k+1)}(y_1)\ll_k\sum_{1\le r\le v}(1+v-r)\frac{((k+1)\log\log
y_1)^r}{(r+1)!}\ll_k\frac{((k+1)\log\log y_1)^v}{(v+1)!},\ee since
$(k+1)\log\rho>1$. It remains to estimate the sum of
$S_r^{(k+1)}(y_1)$ over $r>10kv$. Let $r>10kv$ and $a\in\SN$ so that
$\mu^2(a)=1$ and $\omega_k(a)=r$. Then we may uniquely write
$a=a_1a_2$ with $P^+(a_1)\le k<P^-(a_2)$, in which case
$\omega(a_2)=r$. Applying Lemma \ref{s3l1}(a)\;we find that
$$L^{(k+1)}(a)\le(\log2)^k\tau_{k+1}(a)=(\log 2)^k\tau_{k+1}(a_1)(k+1)^r.$$
Hence
\be\begin{split}\sum_{r>10kv}S_r^{(k+1)}(y_1)&\le(\log2)^k\sum_{r>10kv}\sum_{\substack{P^+(a_1)\le
k\\\mu^2(a_1)=1}}\sum_{\substack{a_2\in\mathscr{P}(k,y_1)\\\omega(a_2)=r,\mu^2(a_2)=1}}\frac{\tau_{k+1}(a_1)(k+1)^r}{a_1a_2}\nonumber\\
&\ll_k\sum_{r>10kv}(k+1)^r\sum_{\substack{a_2\in\mathscr{P}(k,y_1)\\\omega(a_2)=r,\mu^2(a_2)=1}}\frac1{a_2}
\le\sum_{r>10kv}\frac{(k+1)^r}{r!}\left(\sum_{k<p\le
y_1}\frac1p\right)^r.\end{split}\ee Since
$$\sum_{k<p\le y_1}\frac1p=\log\log y_1+O_k(1),$$ we
deduce that
\be\label{s4e9}\begin{split}\sum_{r>10kv}S_r^{(k+1)}(y_1)&\ll_k\frac{((k+1)\log\log
y_1+O_k(1))^{10kv}}{(10kv)!}\ll_k\frac{((k+1)\log\log
y_1)^v}{(v+1)!}.\end{split}\ee Combining relations
\eqref{s4e7},\;\eqref{s4e8}\;and \eqref{s4e9}\;we get that
$$S^{(k+1)}(y_1)\ll_kS^{(k+1)}_0(y_1)+\frac{((k+1)\log\log
y_1)^v}{(v+1)!}\ll_k1+\frac{((k+1)\log\log y_1)^v}{(v+1)!}.$$ Thus
an application of Stirling's formula completes the proof.
\end{proof}

Finally, insert the estimate of Lemma \ref{s4l5}\;into Lemma
\ref{s4l2}\;to finish the proof of the upper bound in Theorem
\ref{th1}.

\section{Estimates from order statistics}\label{ord_stat}

The following discussion is a generalization of estimates about
uniform order statistics obtained in \cite{kf1}\;in order to fit
this context. Our ultimate goal is to give a proof of Lemmas
\ref{s3l9}\;and \ref{s4l4}. Set
$$S_r(u,v)=\left\{(\xi_1,\dots,\xi_r)\in\SR^r:0\le\xi_1\le\cdots\le\xi_r\le1,\;\xi_i\ge\frac{i-u}{v}\;(1\le
i\le r)\right\}$$ and
$$Q_r(u,v)=\mathbf{Prob}\left(\xi_i\ge\frac{i-u}v\;(1\le i\le
k)\, \Big\lvert\,0\le\xi_1\le\cdots\le\xi_r\le1\right)=r!\vol(S_r(u,v)).$$

Combining Theorem 1 in \cite{kf3}\;and Lemma 11.1 in \cite{kf2}\;we
have the following result.

\begin{lemma}\label{s5l1} Let $w=u+v-r$. Uniformly in $u>0$,
$w>0$ and $r\ge1$, we have
$$Q_r(u,v)\ll\frac{(u+1)(w+1)}{r}.$$ Furthermore, if $1\le u\le r$, then
$$Q_r(u,r+1-u)\ge\frac{u-1/2}{r+1/2}.$$
\end{lemma}

Also, we need Lemma 4.3 from \cite{kf1}, which we state below. Note
that we have replaced the constant``10" with a general constant $C$,
$(u+v-r)^2$ by $u+v-r$ and $(g-2)!$ by $(g-1)!$, which is allowed
because we are using Lemma \ref{s5l1}\;in place of Lemma 4.1 in
\cite{kf1}. The proof remains exactly the same.

\begin{lemma}\label{s5l2} Suppose $g,r,s,u,v\in\SZ$ satisfy $$2\le g\le
r/2,\;s\ge 0,\;r\le Cv,\;u\ge 0,\;u+v\ge r+1,$$ where $C$ is a
constant $>1$. Let $R$ be the subset of $\bv{\xi}\in S_r(u,v)$ such
that, for some $l\ge g+1$, we have
$$\frac{l-u}v\le\xi_l\le\frac{l-u+1}v,\;\;\;\xi_{l-g}\ge\frac{l-u-s}v.$$
Then $${\rm
Vol}(R)\ll_C\frac{(C(s+1))^g}{(g-1)!}\frac{(u+1)(u+v-r)}{(r+1)!}.$$
\end{lemma}

We now prove Lemma \ref{s3l9}.

\begin{proof}[Proof of Lemma \ref{s3l9}] For $\bv\xi=(\xi_1,\dots,\xi_B)\in\SR^B$ set
$$F_B(\bv\xi)=\sum_{j=1}^B\lambda^{j-B\xi_j}.$$ Note that
\be\label{s5e100}\begin{split}\vol(Y_B(N))&=\vol(S_B(1,B))-\vol(\{\bv\xi\in
S_B(1,B):F_B(\bv\xi)>\lambda^N\}\\
&\ge\frac1{(2B+1)B!}-\frac1{\lambda^N}\int\limits_{S_B(1,B)}F_B(\bv\xi)d\bv\xi,\end{split}\ee
by Lemma \ref{s5l1}. In \cite[Lemma 4.9, p. 423-424]{kf2}\;it is
shown that
$$\int\limits_{S_r(u,v)}\sum_{j=1}^r2^{j-v\xi_j}d\bv\xi\ll\frac{2^uu}{(r+1)!},$$ provided that $u+v=r+1$, $1\le v\le
r\le100(v-1)$ and $r$ is large enough. Following the same argument
we have
\be\label{s5e101}\int\limits_{S_B(1,B)}F_B(\bv\xi)d\bv\xi\ll_k\frac1{(B+1)!}\ee
(the only thing we need to check is that $\lambda>1$ so that the
integral $\int_0^\infty(y+1)^3\lambda^{-y}dy$ converges). By
\eqref{s5e100}\;and \eqref{s5e101}\;we deduce that
$$\vol(Y_B(N))\ge\frac1{(2B+1)B!}-O_k\left(\frac{\lambda^{-N}}{(B+1)!}\right)\gg\frac1{(B+1)!},$$
provided that $N=N(k)$ is large enough. This completes the proof.
\end{proof}

Finally, we show Lemma \ref{s4l4}. Before we get to this proof we
need a preliminary result. For $\mu>1$ define
$$\mathcal{T}_\mu(r,v,\gamma)=\{0\le\xi_1\le\cdots\le\xi_r\le1:\mu^{v\xi_1}+\cdots+\mu^{v\xi_j}\ge\mu^{j-\gamma}\;(1\le
j\le r)\}.$$ Then we have the following estimate.

\begin{lemma}\label{s5l4} Suppose $r,v,\gamma$ are integers with $\gamma\ge0$ and $1\le r\le
Cv$, where $C>1$ is a constant. Set $b=r-v$ and
$$Y=\begin{cases}b&\text{if}\;b\ge\gamma+1,\cr
(\gamma-b+1)(\gamma+1)&\text{else}.\end{cases}$$ Then
$$\vol(\mathcal{T}_\mu(r,v,\gamma))\ll_{C,\mu}\frac
Y{\mu^{\mu^{b-\gamma}}(r+1)!}.$$
\end{lemma}

\begin{proof} Set $t=\max\{b-\gamma,2+\frac{\log32}{\log\mu}\}$. For every
$\bv\xi\in\mathcal{T}_\mu(r,v,\gamma)$ we have that either
\be\label{s5e1}\xi_j>\frac{j-\gamma-t}v\quad(1\le j\le r)\ee or
there are integers $h\ge t+1$ and $1\le l\le r$ such that
\be\label{s5e2}\min_{1\le j\le
r}\biggl(\xi_j-\frac{j-\gamma}v\biggr)=\xi_l-\frac{l-\gamma}v\in\biggl[\frac
{-h}v,\frac{-h+1}v\biggr].\ee Let $V_1$ be the volume of
$\bv\xi\in\mathcal{T}_\mu(r,v,\gamma)$ that satisfy
\eqref{s5e1}\;and let $V_2$ be the volume of
$\bv\xi\in\mathcal{T}_\mu(r,v,\gamma)$ that satisfy \eqref{s5e2}\;
for some integers $h\ge t+1$ and $1\le l\le r$. First, we bound
$V_1$. If $b\ge\gamma+2+\log32/\log\mu$ so that $t=b-\gamma$, then
\eqref{s5e1}\;is not possible because it would imply that $\xi_r>1$.
So assume that $b<\gamma+2+\log32/\log\mu$, in which case
$t=2+\log32/\log\mu$. Then
$$V_1\le\frac{Q_r(\gamma+2+\frac{\log32}{\log\mu},v)}{r!}
\ll\frac{(\gamma+3+\frac{\log32}{\log\mu})(\gamma+3+\frac{\log32}{\log\mu}-b)}{(r+1)!}
\ll_{\mu}\frac Y{\mu^{\mu^{b-\gamma}}(r+1)!},$$ by Lemma \ref{s5l1}.
Finally, we bound $V_2$. Fix $h\ge t+1$ and $1\le l\le r$ and
consider $\bv\xi\in\mathcal{T}_\mu(r,v,\gamma)$ that satisfies
\eqref{s5e2}. Then
$$-\frac{l-\gamma}v\le\xi_l-\frac{l-\gamma}v\le\frac{-h+1}v$$ and
consequently
$$l\ge\gamma+h-1\ge\gamma+t>2.$$ Set
\be\label{s5e3b}h_0=h-1-\left\lceil\frac{\log4}{\log\mu}\right\rceil\ge
t-\left(\frac{\log4}{\log\mu}+1\right)\ge1+\frac{\log8}{\log\mu}.\ee
We claim that there exists some $m\in\SN$ with $m\ge h_0$ and
$\lfloor\mu^m\rfloor<l/2$ such that
\be\label{s5e3}\xi_{l-\lfloor\mu^m\rfloor}\ge\frac{l-\gamma-2m}v.\ee
Indeed, note that
\be\label{s5e10}\begin{split}\mu^{v\xi_1}+\cdots+\mu^{v\xi_l}\le2\sum_{l/2<j\le
l}\mu^{v\xi_j}&\le2\left(\mu^{v\xi_l}+\sum_{\substack{m\ge
0\\\lfloor\mu^{m}\rfloor<l/2}}\sum_{\lfloor\mu^m\rfloor\le
j<\lfloor\mu^{m+1}\rfloor}\mu^{v\xi_{l-j}}\right)\\
&\le2\left(\mu^{h_0}\mu^{v\xi_l}+\sum_{\substack{m\ge
h_0\\\lfloor\mu^{m}\rfloor<l/2}}(\lfloor\mu^{m+1}\rfloor-\lfloor\mu^m\rfloor)\mu^{v\xi_{l-\lfloor\mu^m\rfloor}}\right).\end{split}\ee
So if \eqref{s5e3}\;failed for all $m\ge h_0$ with
$\lfloor\mu^m\rfloor<l/2$, then
\eqref{s5e2}\;and \eqref{s5e10}\;would imply that
\be\begin{split}\mu^{v\xi_1}+\cdots+\mu^{v\xi_l}&<2\left(\mu^{h_0}\mu^{l-\gamma-h+1}+\sum_{m\ge
h_0}(\lfloor\mu^{m+1}\rfloor-\lfloor\mu^m\rfloor)\mu^{l-\gamma-2m}\right)\nonumber\\
&=2\mu^{l-\gamma}\left(\mu^{-\left\lceil\frac{\log4}{\log\mu}\right\rceil}+(\mu^2-1)\sum_{m\ge
h_0+1}\lfloor\mu^m\rfloor\mu^{-2m}-\lfloor\mu^{h_0}\rfloor\mu^{-2h_0}\right)\\
&\le2\mu^{l-\gamma}\left(\frac14+\frac{\mu^{h_0+1}+1}{\mu^{2h_0}}\right)
\le2\mu^{l-\gamma}\left(\frac14+\frac2{\mu^{h_0-1}}\right)\le\mu^{l-\gamma},\end{split}\ee
by \eqref{s5e3b}, which is a contradiction. Hence \eqref{s5e3}\;does
hold and Lemma \ref{s5l2}\;applied with $u=\gamma+h$,
$g=\lfloor\mu^m\rfloor$ and $s=2m$ implies that
\be\begin{split}V_2&\ll_C\sum_{h\ge t+1}\sum_{m\ge h_0}\frac{(\gamma+h+1)(\gamma+h-b)}{(r+1)!}\frac{(C(2m+1))^{\lfloor\mu^m\rfloor}}{(\lfloor\mu^m\rfloor-1)!}\nonumber\\
&\ll_{C,\mu}\sum_{h\ge t+1}\sum_{m\ge h_0}\frac{(\gamma+h+1)(\gamma+h-b)}{(r+1)!}\left(\frac{Ce(2m+1)}{\mu^m}\right)^{\mu^m}\mu^{2m}\nonumber\\
&\ll_{C,\mu}\sum_{h\ge t+1}\sum_{m\ge
h_0}\frac{(\gamma+h+1)(\gamma+h-b)}{\mu^{\mu^{m+m_0}}(r+1)!},\end{split}\nonumber\ee
where $m_0=1+\lceil\log4/\log\mu\rceil$. The sum of
$\mu^{-\mu^{m+m_0}}$ over $m\ge h_0$ is $\ll_\mu\mu^{-\mu^{h}}$.
Finally, summing over $h\ge t+1$ gives us that
$$V_2\ll_{C,\mu}\frac{(\gamma+t+2)(\gamma-b+t+1)}{\mu^{\mu^{t+1}}(r+1)!}
\ll_{C,\mu}\frac Y{\mu^{\mu^{b-\gamma}}(r+1)!},$$ which completes
the proof.
\end{proof}

\begin{proof}[Proof of Lemma \ref{s4l4}] Recall that $\rho=(k+1)^{1/k}$.
Set $$F(\bv\xi)=\biggl(\min_{0\le j\le
r}\rho^{-j}(\rho^{v\xi_1}+\cdots+\rho^{v\xi_j}+1)\biggr)^k$$ and
note that $F(\bv\xi)\le1$. Fix an integer $m\ge 1$. Consider
$\bv\xi\in\SR^r$ with $0\le\xi_1\le\cdots\le\xi_r\le1$ such that
$2^k(k+1)^{-m}\le F(\bv\xi)<2^k(k+1)^{1-m}$. For $1\le j\le r$ we
have that
$$\rho^{-j}(\rho^{v\xi_1}+\cdots+\rho^{v\xi_j})\ge\rho^{-j}$$ and
\be\begin{split}\rho^{-j}(\rho^{v\xi_1}+\cdots+\rho^{v\xi_j})&=\rho^{-j}(\rho^{v\xi_1}+\cdots+\rho^{v\xi_j}+1)-\rho^{-j}\nonumber\\
&\ge\left(F(\bv\xi)\right)^{1/k}-\rho^{-j}\ge2\rho^{-m}-\rho^{-j}.\nonumber\end{split}\ee
Thus
$$\rho^{-j}(\rho^{v\xi_1}+\cdots+\rho^{v\xi_j})\ge\max\{\rho^{-j},2\rho^{-m}-\rho^{-j}\}
\ge\rho^{-m},$$ that is $\bv\xi\in\mathcal{T}_\rho(r,v,m)$.
Therefore Lemma \ref{s5l4}\;applied with $C=10k$ implies that
\be\label{s5e200}\begin{split}U_r(v;k)&\le\sum_{m=1}^\infty2^k(k+1)^{1-m}\vol(\mathcal{T}_\rho(r,v,m))\\
&\ll_k\frac1{(r+1)!}\left(\sum_{1\le m\le b}\frac
b{(k+1)^m\rho^{\rho^{b-m}}}+\sum_{m\ge\max\{1,b\}}\frac{m(m-b+1)}{(k+1)^m}\right).\end{split}\ee
If $b\ge1$, then each sum in the right hand side of
\eqref{s5e200}\;is $\ll_kb(k+1)^{-b}$. On the other hand, if
$b\le0$, then the first sum is empty and the second one is
$\ll_k1+|b|$. In any case,
$$U_r(v;k)\ll_k\frac1{(r+1)!}\frac{1+|b|}{(k+1)^b+1},$$ which completes the proof.
\end{proof}

\bigskip\bigskip

\bibliographystyle{amsplain}

\end{document}